\documentclass[12pt,reqno]{article}
\usepackage[body={6.5in,9.0in},left=1in,top=1in,centering]{geometry}
\usepackage{graphicx}
\usepackage{amssymb}
\usepackage{epstopdf}
\usepackage{amssymb,amsmath,amsthm,amsfonts,color}
\usepackage[colorlinks=true,urlcolor=blue,linkcolor=blue,citecolor=blue,pdfstartview=FitH]{hyperref}
\usepackage{xcolor}
\usepackage{authblk,pstricks}
\usepackage{tikz}
\usepackage[sort,comma]{natbib}
\usepackage[title]{appendix}
\usepackage{mathrsfs}
\newcmykcolor{orange}{0 .61 0.87 0}
\newcmykcolor{navyblue}{0.94 0.74 0 0}
\newcmykcolor{peach}{0 0.50 0.70 0}

\newtheorem{thm}{Theorem}[section]

\newtheorem{cor}[thm]{Corollary}

\newtheorem{prop}[thm]{Proposition}
\newtheorem{lem}[thm]{Lemma}
\theoremstyle{definition}
\newtheorem{defn}[thm]{Definition}

\newtheorem{cnd}[thm]{Condition}

\newtheorem{rem}[thm]{Remark}

\newcommand{\thmref}[1]{Theorem~{\rm \ref{#1}}}
\newcommand{\lemref}[1]{Lemma~{\rm \ref{#1}}}

\newcommand{\cndref}[1]{Condition~{\rm \ref{#1}}}
\newcommand{\propref}[1]{Proposition~{\rm \ref{#1}}}

\newcommand{\remref}[1]{Remark~{\rm \ref{#1}}}

\newcommand{\figref}[1]{Figure~{\rm \ref{#1}}}

\newcommand{\sectref}[1]{Section~{\rm \ref{#1}}}

\makeatletter \@addtoreset{equation}{section}

\allowdisplaybreaks

\def\R{\ensuremath {\mathbb R}}
\newcommand{\EE}{\mathbb E}
\newcommand{\PP}{\mathbb P}

\newcommand{\NN}{\mathbb N}
\newcommand{\RR}{\mathbb R}

\newcommand{\I}{{\mathcal I}}
\newcommand{\A}{{\mathcal A}}
\newcommand{\B}{{\mathcal B}}

\newcommand{\yzstar}{{y_0^*}}
\newcommand{\zzstar}{{z_0^*}}
\newcommand{\e}{\varepsilon}

\newcommand{\wdh}{\widehat}
\newcommand{\wdt}{\widetilde}
\newcommand{\ovl}{\overline}

\newcommand{\lbar}{\overline}
\newcommand{\cadlag}{c\'{a}dl\'{a}g }

\renewcommand{\hat}{\widehat}

\newcommand{\E}{{\cal E}}

\newcommand{\comment}[1]{} 

\newcommand{\vareps}{\varepsilon}

\title{Single-Item Continuous-Review Inventory Models with Random Supplies\thanks{This research was supported in part by the Simons Foundation (grant award numbers  246271 and  523736) and  a DIG award from the University of Wisconsin-Milwaukee.}}
 
\author[1]{K.L. Helmes}
\author[2]{R.H. Stockbridge}
\author[2]{C. Zhu}
\affil[1]{\small Institute for Operations Research, Humboldt University of Berlin, Spandauer Str. 1, 10178, Berlin, Germany, {\tt helmes@wiwi.hu-berlin.de}}
\affil[2]{Department of Mathematical Sciences,   University of Wisconsin-Milwaukee,   Milwaukee, WI 53201,   USA,   {\tt stockbri@uwm.edu}, {\tt zhu@uwm.edu}}
 
\date{~}

\begin{document}
\maketitle

\begin{abstract}
This paper analyzes single-item continuous-review inventory models with random supplies in which the inventory dynamic between orders is described by a diffusion process, and a long-term average cost criterion is used to evaluate decisions. The class of models have general drift and diffusion coefficients and boundary points that are consistent with the notion that demand should tend to reduce the inventory level. Random yield is described by a (probability) transition function which depends on the inventory-on-hand and the {\em nominal} amount ordered; it is assumed to be a distribution with support in the interval determined by the order-from and the nominal order-to locations of the stock level.  Using weak convergence arguments involving average expected occupation and ordering measures, conditions are given for the optimality of an $(s,S)$ ordering policy in the general class of policies with finite expected cost.  The characterization of the cost of an $(s,S)$-policy as a function of two variables naturally leads to a nonlinear optimization problem over the stock levels $s$ and $S$ and existence of an optimizing pair $(s^*,S^*)$ is established under weak conditions. Thus, optimal policies of inventory models with random supplies can be (easily) numerically computed.  The range of applicability of the optimality result is illustrated on several inventory models with random yields. 

\vspace{5 mm}
	
\vspace{3 mm}


\end{abstract}

\section{Introduction}
This paper analyzes a continuous-review inventory management problem when the stock level process is a diffusion with deficient supply; a long-term average cost criterion is used.  The control over the inventory levels is through the action of ordering additional {\em nominal} stock which then results in a random yield of whatever has been ordered. We identify sufficient conditions for optimality of an $(s,S)$ ordering policy in the most general class of admissible policies. 

We model the inventory processes (in the absence of orders) as solutions to a stochastic differential equation
\begin{equation} \label{dyn}
dX_0(t) = \mu(X_0(t))\, dt + \sigma(X_0(t))\, dW(t), \qquad X_0(0) = x_0,
\end{equation}
taking values in an interval ${\mathcal I} = (a,b)$; negative values of $X_0(t)$ represent back-ordered inventory. The detailed discussion in \cite{chen:10} 
validates state-dependent diffusion models for inventory management.

Following the classical approach in inventory theory, an ordering policy $(\tau,O)$ for a model with random supplies is a sequence of pairs $\{(\tau_k,O_k): k \in \NN\}$ in which $\tau_k$ denotes the (random) time at which the $k$th order is placed and $O_k$ denotes its (nominal) size.  The random supply is modelled by the random slack $\Theta$ which is a sequence in which, for each $k$, $\Theta_k$ gives the deficit of the quantity delivered from the order amount $O_k$;  it also represents the deficiency between the intended inventory level and the actual level after the order delivery.  While the order quantities $\{O_k\}$ are determined by the decision maker, the corresponding slack variables $\{\Theta_k\}$ arise from factors involving the supplier.  The inventory level process $X$ resulting from an ordering policy $(\tau,O)$ and corresponding slack $\Theta$ therefore satisfies the equation
\begin{equation} \label{controlled-dyn}
X(t) = x_0 + \int_0^t \mu(X(s))\, ds + \int_0^t \sigma(X(s))\, dW(s) + \sum_{k=1}^\infty I_{\{\tau_k \leq t\}} (O_k-\Theta_k).
\end{equation}
Note, the initial inventory level $X(0-)=x_0$ may be such that an order is placed at time $0$ resulting in a new inventory level at time $0$; this possibility occurs when $\tau_1 = 0$.  Also observe that $X(\tau_k-)$ is the inventory level just prior to the $k$th order being placed while $X(\tau_k)=X(\tau_k-)+O_k-\Theta_k$ is the level with the new inventory.  Thus, this model assumes that orders are filled instantaneously.  \sectref{sect:form} describes the inventory process $X$ more formally as an impulse controlled diffusion process and adopts a different formulation of a nominal ordering policy $(\tau,Z)$ in which $Z=\{Z_k\}$ denotes the nominal inventory levels following (non-deficient) orders.

For the time being, continuing with the informal description above, let $(\tau,O)$ be an ordering policy, $\Theta$ be the corresponding slack and let $X$ be the resulting inventory level process satisfying \eqref{controlled-dyn}. Let $c_0$ and $c_1$ denote the holding/back-order cost rate and (nominal) ordering cost functions, respectively.  We assume there is some constant $k_1 > 0$ such that $c_1 \geq k_1$; this constant represents the fixed cost for placing each order.  The long-term average expected holding/back-order plus ordering costs to be analyzed is
\begin{equation} \label{lta-obj-fn}
J:= \limsup_{t\rightarrow \infty} t^{-1} \EE\left[\int_0^t c_0(X(s))\, ds + \sum_{k=1}^\infty I_{\{\tau_k \leq t\}} c_1(X(\tau_k-),X(\tau_k))\right];
\end{equation}
the expectation is with respect to all random factors involved in the model. The goal is to identify an ordering policy so as to minimize the cost.
For models with random supplies there are other more exotic cost structures that can be considered. The use of $X(\tau_k)$ in the cost functional \eqref{lta-obj-fn} captures the situation:  ``you pay for what you get"; see the paragraph following \cndref{cost-cnds} for further details.

As we have mentioned earlier, we study a generalization of the problem examined in \cite{helm:18}. In particular, we refer the reader to that paper and to \cite{helm:15b} for a discussion of the existing literature related to the problem with {\em non-deficient} supplies in which $\Theta_k=0$ for all $k$; see also \cite{bens:11} and references therein. As far as problems with {\em random} yield are concerned, the papers \cite{yano_lee:95} and \cite{tina:17} provide excellent reviews of such single-item continuous-review inventory models. In particular, the \cite{yano_lee:95} survey paper offers an extensive account of how various yield distributions and cost structures arise in practical applications. Papers by \cite{fed_zip:86} and \cite{zhe_fed:91} are more technical in nature. Furthermore, \cite{fed_zip:86} explicitly addresses the optimality of $(s,S)$-policies for a special continuous-review model with random supplies. The paper \cite{zhe_fed:91} is most useful since it decribes an efficient algorithm for computing optimal $(s,S)$-policies and applies to both periodic-review and continuous-review inventory systems. The paper \cite{bpp:94} considers a continuous review problem with (proportional) random yield.   The authors use renewal theory to analyze their inventory model, which is also used in this paper. Among the many other papers devoted to inventory problems with random yield, we would like to point to the publications \cite{ind_tra:07}, \cite{ind_vog:13} and \cite{song_wang:17}.  These papers analyze {\em periodic} review problems and nicely describe the challenges due to the presence of (uniformly distributed) random supply.  The paper \cite{sato:18} analyzes an infinite-horizon discounted cost criterion for a distributor when the supplier has uncertain production.  Further, it considers both the  supplier's and the distributor's problems, showing that coordinized decision making results in reduced expected costs.

Irrespectively of the many different models that have been considered in the literature, a common theme which lurks in the background of all the papers devoted to random yield is the quest to identify either an optimal or at least a nearly optimal order strategy. In some publications, the thrust is to propose and justify a heuristic policy, assuming that an optimal order policy has a particular (simple) structure. Instead, we formulate general conditions on the model under which an $(s,S)$ policy is optimal for the long-term average criterion.

This paper extends to the case with {\em random yield} our examination of inventory models of diffusion type with {\em non-deficient supplies} in \cite{helm:18}.   Even though the same approach is used in these two papers, the analyses are more technical in the present manuscript due to the inclusion of random supplies satisfying \cndref{Q-cnds}.  For example, a {\em Minimum Delivery Guaranty}\/ condition is required for the existence of a valid mathematical model, a point that has been overlooked in the literature; see for example \cite{korn:97}.  Also an {\em Assured Supply Commitment}\/ condition is essential to our proof of optimality of an $(s,S)$ policy; see \thmref{G0-feasible}.  Furthermore, \cndref{extra-cnd} of this paper removes a monotonicity requirement in Condition~2.3 of \cite{helm:18}, allowing the results to apply to a larger class of models.

This paper is organized as follows.  The next section formulates the problem; in particular, it states conditions on the family of random yield measures that are key to the existence of a mathematical model for continuous-time inventory management as well as the optimization results.  It further introduces two important functions and adapts some results from \cite{helm:18} to the model having random supply.  It culminates with the main existence result in \thmref{F-optimizers}.  Sections \ref{sect:occ-meas} and \ref{sect-main result} briefly discuss the expected occupation and ordering measures, adapted for models with random yield, and an auxiliary function $U_0$, which are at the heart of the analysis.  \sectref{sect:optim} then establishes the optimality of an $(s,S)$ policy within the much larger class of admissible nominal ordering policies.  The main optimality result is in \thmref{G0-feasible}; its proof is broken into several parts which precede it.  The paper concludes with a discussion of three examples in Section \ref{sect:examples} which indicate the reliance of Theorems \ref{F-optimizers} and \ref{G0-feasible} in obtaining an optimal ordering policy.

\section{Formulation and Existence Result} \label{sect:form}
This section briefly establishes the models under consideration which generalize those studied in \cite{helm:18}. While the general approach is very similar to the one taken in that paper, special care must be taken pertaining to the formulation of the random yield, the cost structure, the definition of the (nominal) occupation measure, the  particular jump operators and the proofs of several results. The differences between the two papers will be highlighted in the following sections. For a detailed discussion of the dynamics of the underlying uncontrolled diffusion and its boundary behavior we refer the reader to \cite{helm:18} and to Chapter 15 of \cite{karl:81}. The latter reference is particularly useful when checking properties of the scale function and the speed measure; both concepts are used in the definition of functions in \sectref{important-fns}.  

\subsection{Formulation of the model}
Let ${\mathcal I} = (a,b) \subseteq \R$.  In the absence of ordering, the inventory process $X_0$ satisfies \eqref{dyn} and is a regular diffusion. Throughout the paper we assume that the functions $\mu$ and $\sigma$ are continuous on $\I$, and that \eqref{dyn} is nondegenerate. The initial position of $X_0$ is taken to be $x_0$ for some $x_0 \in {\cal I}$.  We place the following assumptions on the underlying diffusion model.

\begin{cnd} \label{diff-cnd}
\begin{description}
\item[(a)] Both the speed measure $M$ and the scale function $S$ of the process $X_0$ are absolutely continuous with respect to Lebesgue measure.
\item[(b)] The left boundary $a$ is attracting and the right boundary $b$ is non-attracting.  Moreover, when $b$ is a natural boundary, $M[y,b) < \infty$ for each $y \in \I$.  The boundaries $a = -\infty$ and $b = \infty$ are required to be natural.
\end{description}
\end{cnd}

Associated with the scale function $S$ of \cndref{diff-cnd}, one can define the scale measure on the Borel sets of ${\mathcal I}$ by $S[y,v] = S(v) - S(y)$ for $[y,v] \subset {\mathcal I}$. From the modeling point of view, \cndref{diff-cnd}(b) is reasonable since it essentially says that, in the absence of ordering, demand tends to reduce the size of the inventory. The boundary point $a$ may be regular, exit or natural with $a$ being attainable in the first two cases and unattainable in the third. In the case that $a$ is a regular boundary, its boundary behavior must also be specified as being either reflective or sticky. The boundary point $b$ is either natural or entrance and is unattainable from the interior in both cases. Following the approach in \cite{helm:15b} and \cite{helm:18}, we define the state space of possible inventory levels to be the interval ${\cal E}$ which excludes any natural boundary point; it includes $a$ when $a$ is attainable, and $b$ when it is entrance. Since orders typically increase the inventory level, define ${\cal R} = \{(y,z) \in {\cal E}^2: y < z\}$, the set of states cross the set of feasible actions (in a particular state), in which $y$ denotes the pre-order inventory level and the control value $z$ is the {\em nominal} post-order level.  The {\em actual} post-order inventory level will be determined by $y$, $z$ and the realization of the slack variable of the associated order size; explained differently, the  post-order inventory level is given as the realization of a transition function $Q(\cdot\,;y,z)$ which depends on $(y,z)$. 

Since we are using weak convergence methods for measures on ${\cal E}$ and ${\cal R}$, we will need the closures of these sets as well. Define $\overline{\mathcal{E}}$ to be the {\em closure in $\RR$}\/ of $\mathcal{E}$; thus when a boundary is finite and natural, it is not an element of $\mathcal{E}$ but is in $\overline{\mathcal{E}}$. Note $\pm \infty \notin \overline{\cal E}$. Also set $\overline{\cal R} = \{(y,z) \in {\cal E}^2: y \leq z\}$; in contrast to ${\cal R}$, the set $\overline{\cal R}$ includes orders of size $0$. Notice the subtle distinction between $\overline{\cal E}$ which includes boundaries that are finite and natural and $\overline{\cal R}$ which does not allow either coordinate to be such a point.  


The random yields are determined by the family ${\cal Q} = \{Q(\cdot\,;y,z) : (y,z) \in \ovl{\cal R}\}$ of probability measures parametrized by $(y,z)\in \ovl{\cal R}$ such that (i) $Q(\,\cdot\,;y,z)$ is a probability measure for each $(y,z)\in \ovl{\cal R}$ and (ii) for each $E\in \B(\E)$, $(y,z) \to Q(E;y,z)$ is measurable.  $Q$ is a transition function on $\E\times \ovl{\cal R}$.  The probability measure  $Q(\cdot;  y,z)$ is the distribution for the resulting inventory level following an order of size $z-y$.  We further impose  support, continuity and supply requirements on this family.
\begin{cnd} \label{Q-cnds}
The collection ${\cal Q}$ satisfies
\begin{description}
\item[(a)] 
\begin{itemize}
\item[(i)] for each $y \in {\cal E}$, $Q(\cdot\,;y,y) = \delta_{\{y\}}(\cdot)$;
\item[(ii)] for each $(y,z) \in {\cal R}$, supp$(Q(\cdot\,;y,z)) \subset (y,z]$; 
\end{itemize}
\item[(b)] for each $(y,z) \in \ovl{\cal R}$, for  any sequence $\{(y_n,z_n)\in \ovl{\cal R}:n\in \NN\}$ with $y_n\rightarrow y$ and $z_n\rightarrow z$ as $n\rightarrow \infty$, the measures $Q(\cdot;y_n,z_n)$ converge weakly to $Q(\cdot;y,z)$ as $n\rightarrow \infty$; this weak convergence is denoted by $Q(\cdot\,;y_n,z_n) \Rightarrow Q(\cdot\,;y,z)$; and
\item[(c)] 
when $b$ is a natural boundary, for each $[d_1,d_2]\subset \I$, there exists  a $\delta > 0$ so that for each $\wdt{z}_1$ with $d_2 < \wdt{z}_1 < b$,
\setcounter{equation}{0}
\begin{equation} \label{e:new-VIP}
\liminf_{z\to b} \inf_{y\in [d_1,d_2]} Q((\wdt{z}_1,b);y,z) \geq \delta.
\end{equation}
\end{description}
\end{cnd}




Condition~(a,i) indicates that an active order of nominal size $0$ will not change the inventory level.  Condition~(a,ii) implies the existence of a {\em Minimal Delivery Guarranty} (MDG) that, with probability $1$, assures a fixed positive amount (up to the amount ordered) will be delivered when a positive nominal amount is ordered.  This condition is essential to show that each admissible policy, including $(s,S)$ policies, has a valid mathematical model for random supplies (cf.~Definition~2.3 and following comments in \cite{helm:19}).  The fact that this kind of a condition needs to be imposed on inventory models with random supply to have a proper mathematical model of the controlled process has been overlooked in the literature.  Condition~(b) requires continuity of the mapping $Q$ in the topology of weak convergence.  Condition~(c) is an {\em Assured Supply Commitment}\/ (ASC) that can be interpreted to be a ``V(ery) I(mportant) C(ustomer)'' condition in the sense that a customer who nominally orders to very high levels of inventory has a significant likelihood of receiving almost all of his order.  This condition is used to establish the existence of an optimizer in \thmref{F-optimizers}  
and to establish the optimality of a nominal $(s,S)$ policy in Section~\ref{sect:optim}. 

We illustrate how \cndref{Q-cnds} may be satisfied when $b=\infty$, a natural boundary.  For fixed $0 < \Delta < 1$, let $\wdt{Q} \in {\cal P}[\Delta,1]$ be fixed.  For $(y,z) \in {\cal R}$, let $f_{(y,z)}: [0,1] \rightarrow {\cal E}$ be the linear mapping with $f_{(y,z)}(0)=y$ and $f_{(y,z)}(1)=z$.  Then the family ${\cal Q}$ defined for $(y,z) \in {\cal R}$ by $Q(\cdot\,;y,z) = \wdt{Q}f_{(y,z)}^{-1}(\cdot)$ always satisfies \cndref{Q-cnds}.  A special case of this family occurs when $\wdt{Q}$ is the uniform distribution on $[\Delta,1]$, resulting in a continuous review inventory model with {\em nearly stochastically proportional yields}.  A second special case having $\wdt{Q}(\cdot)=\delta_{\{1\}}(\cdot)$ corresponds to the slack being $0$ and therefore models non-deficient supply.  { Further examples will be examined in Section~\ref{sect:examples}, for example when $b$ is a finite natural boundary.} 

It will be important throughout the manuscript to average functions using transition functions.  For a measurable function $\ell$ on $\ovl{\cal R}$ and a transition function $Q$, we adopt the shorthand notation
\begin{equation} \label{hat-def}
\wdh \ell(y,z): = \int \ell(y,v) Q(dv; y,z), \qquad (y,z) \in \ovl{\cal R},
\end{equation}
with the understanding that the integral exists in $\ovl{\RR}$.

Turning to the cost functions, we impose the following standing assumptions throughout the paper.
\begin{cnd} \label{cost-cnds}
\begin{itemize}
\item[(a)] The holding/back-order cost function $c_0: {\mathcal  I} \rightarrow \R^+$ is  continuous.  Moreover, at the boundaries 
$$\lim_{x\rightarrow a} c_0(x) =: c_0(a) \mbox{  exists in $\overline{\R^+}$ and } \lim_{x\rightarrow b} c_0(x) =: c_0(b) \mbox{  exists in $\overline{\R^+}$;}$$ 
we require $c_0(\pm\infty) = \infty$.  Finally, for each $y \in {\cal I}$,
\begin{equation} \label{c0-M-integrable}
\int_y^b c_0(v)\, \, dM(v)  < \infty.
\end{equation}
\item[(b)] The function $c_1:\overline{\cal R} \rightarrow \overline{\R^+}$ is in $C(\overline{\mathcal R})$ with $c_1 \geq k_1 > 0$ for some constant $k_1$.
\end{itemize}
\end{cnd}

The function $c_1$ is the building block for more complex cost structures of models with random supplies. For example, in the case when the decision maker ``pays for what he orders'' the ordering cost function is $c_1$ itself.  When the cost stucture is ``you pay for what you get'', the function $\hat{c}_1$ is used. For the remainder of the main sections, we analyze the inventory problem using $\hat{c}_1$, i.e. we pay for what we get; see also the following subsection. 

We adapt to this inventory application the model constructed in \cite{helm:19} for impulse-controlled processes having processes that are continuous between impulses.  The model is built on an augmentation of the space $D_\E[0,\infty)$ of \cadlag paths from $[0,\infty)$ to $\E$ using the natural filtration $\{{\cal F}_t\}$ in which $X$ is the coordinate process and ${\cal F}_t = \sigma(X(s): 0 \leq s \leq t)$. 

We now define a nominal ordering policy.  In order to do so, we need to specify the filtration of information used by the decision maker to determine the jump-from locations and the nominal jump-to-locations of a policy.  Let $\{{\cal F}_{t-}\}$ be given by ${\cal F}_{t-} = \sigma(X(s): 0 \leq s < t)$ for $t> 0$ with ${\cal F}_{0-}=\sigma(X(0-))$ being the $\sigma$-algebra generated by the inventory level prior to any intervention at time $0$.  It is also important to specify the $\sigma$-algebra of information available prior to a stopping time.  Let $\eta$ be an $\{{\cal F}_{t-}\}$-stopping time.  The $\sigma$-algebra ${\cal F}_{\eta-} := \sigma(\{A\cap \{\eta > t\}:  A\in {\cal F}_t, t\geq 0\})$. 

For the inventory management problem with random supply, the class ${\cal A}$ of admissible nominal ordering policies $(\tau,Z) =\{(\tau_{k}, Z_{k}), k\in \mathbb N \}$ is defined as follows:
\begin{description}
\item[(i)] $\{\tau_k:k\in \NN\}$ is a {\em strictly increasing} sequence of $\{{\cal F}_{t-}\}$-stopping times with $\tau_k \to \infty$; 
\item[(ii)] for each $k \in \NN$, $Z_k \in \E$ is ${\cal F}_{\tau_{k}-}$-measurable with $Z_k > X(\tau_k-)$; and 
\item[(iii)] the cost \eqref{lta-obj-fn} is finite and is denoted by $J(\tau,Z)$; note the inclusion of the policy in the notation. 
\end{description}
The requirement that the sequence $\{\tau_k\}$ be strictly increasing implies that at most one order can be placed at any time while the use of $\{{\cal F}_{t-}\}$ prevents the ordering decisions from knowing the supplied amount when an order is placed.  The random variable $Z_k$ in (ii) is the nominal order-to location so is the value $X(\tau_k-)+O_k$ when $O_k$ denotes the nominal order size.  The construction in \cite{helm:19} uses the measure $Q(\cdot;X(\tau_k-),Z_k)$ to select the actual random supply inventory level $X(\tau_k)$ at time $\tau_k$.  Hence the corresponding random slack is $\Theta_k = Z_k - X(\tau_k)$.

Thus, given the transition functions ${\cal Q}$ and an admissible policy in the class ${\cal A}$, the associated inventory process $X$ will be a jump-diffusion process characterized by the generator of the process $X_0$, the jump operator determined by the the decision of ordering up to a nominal level $z$ and the transition function $Q$.  

Looking at the infinitesimal behavior, the generator of the process $X$ between jumps (corresponding to the diffusion $X_0$) is $Af = \mbox{$\frac{\sigma^2}{2}$} f'' + \mu f'$, which is defined for all $f \in C^2({\mathcal I})$; equivalently, $Af = \frac{1}{2} \frac{d~}{dM} \left(\frac{df}{dS}\right)$. The effects that ordering and random yields have on the inventory process and its expected cost will be defined by the jump operator $B: C({\cal E}) \rightarrow C(\overline{\mathcal{R}})$, $Bf(y,z) := f(z) - f(y)$ for $(y,z) \in \overline{\cal R}$ for an order with non-deficient supply having transition function $Q(\cdot\,;y,z)=\delta_{z}(\cdot)$, and for the case of random yield by the $\hat{~}$-operation $\hat{Bf}(y,z) := \,  \int Bf(y,v)\, Q(dv;y,z)$ when the order-from location is $y$ and action $z$ selects a transition function $Q(\cdot\,;y,z)$.  

\subsection{Important functions} \label{important-fns}

As in \cite{helm:18}, the following two functions play a central role in our search for an optimal ordering policy.  Recall, $M$ denotes the speed measure and $S$ represents the scale measure.  Using the initial position $x_0\in \mathcal{I}$, define the functions $g_0$ and $\zeta$ on $\I$ by
\begin{align} \label{g0-fn}
  g_{0}(x) := \int_{x_0}^{x} \int_{u}^{b} 2 c_{0}(v)\, dM(v)\, dS(u) \qquad \text{and} \qquad
\zeta(x) := \int_{x_0}^{x} \int_{u}^{b} 2\, dM(v)\, dS(u) ,
\end{align}
and extend these functions to $\overline{\cal E}$ by continuity.  Observe that both $g_0$ and $\zeta$ are negative on $(a,x_0)$ and positive on $(x_0,b)$; also $g_0$ may take values $\pm \infty$ at the boundaries while $\zeta$ is $\pm \infty$ for natural boundaries. Using the second characterization of $A$, it immediately follows that $g_0$ and $\zeta$, respectively, are particular solutions on $\mathcal{I}$ of
\begin{equation} \label{inhom-eq}
\left\{\begin{array}{l}
Af = - c_0, \\
f(x_0) = 0,
\end{array} \right. \quad \mbox{ and } \quad \left\{\begin{array}{l}
Af = -1, \\
f(x_0) = 0.
\end{array}\right.
\end{equation}
Other solutions to these differential equations having value $0$ at $x_0$ include summands of the form $K(S(x) - S(x_0))$, $K \in \RR$, since the constant function and the scale function $S$ are linearly independent solutions of the homogeneous equation $Af = 0$.  However, such additional terms grow too quickly near the boundary $b$ so that the transversality condition \eqref{G0-transv-cnd} in \propref{sS-in-Q2} below fails (see Remark~4.2 of \cite{helm:18}) and therefore the definitions of $g_0$ and $\zeta$ in \eqref{g0-fn} exclude these terms.

To gain some intuition for the functions $g_0$ and $\zeta$, let $y, v \in {\cal E}, y < v$, and let $X_0$ satisfy \eqref{dyn} with $X_0(0) = v$. 
Define $\tau_{v,y} := \inf\{t \geq 0: X_0(t) = y\}$. 
Then, Proposition~2.6 in \cite{helm:15b} shows that 
$$\EE_v\left[\int_0^{\tau_{v,y}} c_0(X_0(s))\, ds\right] =  {B}g_0(y,v),\qquad \mbox{and}\qquad \EE_v[\tau_{v,y}] = {B}\zeta(y,v),$$ 
and a simple extension establishes that if $X_0(0) \sim Q(\cdot \ ;y,z)$, for $(y,z) \in {\cal R}$, then 
$$\int \EE_v\left[\int_0^{\tau_{v,y}} c_0(X_0(s))\, ds\right]Q(dv;y,z) = \hat{Bg_0}(y,z),\,\,\, \mbox{and}\,\,\, \int \EE_v[\tau_{v,y}]Q(dv;y,z)= \hat{B\zeta}(y,z).$$

The proof of our basic existence result, \thmref{F-optimizers}, relies on the asymptotic behavior of the functions $c_0$, $g_0$ and $\zeta$  when the boundaries are natural.
The following lemma, whose proof can be found in Lemma~2.1 of \cite{helm:18}, summarizes such asymptotic behavior.
 
\begin{lem} \label{g0-at-a}
Assume \cndref{diff-cnd}.  Suppose $a$ and $b$ are natural boundaries and let $c_0(a)$ and $c_0(b)$ be as in \cndref{cost-cnds}(a).  Then the following asymptotic behaviours hold:
\begin{align} \label{z-fixed-g0-zeta-diff-at-a}
& \lim_{y\rightarrow a} \frac{Bg_0(y,v)}{B\zeta(y,v)} = c_0(a), \quad \forall v\in \mathcal{I}; \  & & \lim_{v\rightarrow b} \frac{Bg_0(y,v)}{B\zeta(y,v)}  = c_0(b), \quad \forall y\in \mathcal{I}; & \\
\label{double-g0-zeta-diff-at-a}
& \lim_{(y,v)\rightarrow (a,a)} \frac{Bg_0(y,v)}{B\zeta(y,v)} = c_0(a); \quad & & \lim_{(y,v) \rightarrow (b,b)} \frac{Bg_0(y,v)}{B\zeta(y,v)}   = c_0(b); & \\
\label{g0-zeta-ratio-at-a}
& \lim_{y\rightarrow a} \frac{g_0(y)}{\zeta(y)} = c_0(a); \quad & &  \lim_{v\rightarrow b} \frac{g_0(v)}{\zeta(v)}   = c_0(b),& 
\end{align}
implying $\lim_{y\rightarrow a} g_0(y) = -\infty$ when $c_0(a) > 0$ and $\lim_{v\rightarrow b} g_0(v) = \infty$ when $c_0(b) > 0$.
\end{lem}

Another function of importance to the solution of the problem is $\wdh{c}_1$, which we remind the reader is defined to be $\wdh{c}_1(y,z)=\int c_1(y,v) Q(dv;y,z)$,  where $(y,z) \in \overline{\cal R}$.  The first proposition indicates a difference between the properties of the ordering cost structure of the random supply model and the model with non-deficient deliveries.

\begin{prop} \label{c1hat-lsc}
Assume Conditions \ref{diff-cnd} - \ref{cost-cnds}.  Then $\wdh{c}_1$ is lower semicontinuous.
\end{prop}

\begin{proof}
We need to show that for every $(y,z) \in {\overline{\cal R}}$ and every sequence $\{(y_n,z_n): n\in\NN\}$ in ${\overline {\cal R}}$ which converges to $(y,z)$,
\begin{equation}\label{eq-c1hat-lsc}
 \wdh{c}_1(y,z)   \le  \liminf_{n\rightarrow \infty} \wdh{c}_1(y_n,z_n).
\end{equation} 
We may assume that the function $c_1$ is bounded; the monotone convergence theorem implies the inequality \eqref{eq-c1hat-lsc} for unbounded cost functions once it has been established for a truncated form of  $c_1$.
To verify \eqref{eq-c1hat-lsc} we shall rely on the elementary but most useful Lemma~2.1 in \cite{serf:82}. In the sequel, we verify the hypothesis of this lemma. To this end, for the given pair $(y,z)$ and the points $y_n, n \in \NN$, we define nonnegative continuous functions $f$ and $f_n$ on $\cal E$ as follows. For $v \in \cal E$,  let
\begin{equation}\label{eq-f-fn}
f(v) := \left\{\begin{array}{cl}\displaystyle 
{c_1}(y,v), & \quad v \ge y, \rule[-15pt]{0pt}{15pt}\\
{c_1}(y,y), & \quad v \le y; 
\end{array} \right.
 \text{and} \quad
f_n(v) := \left\{\begin{array}{cl}\displaystyle 
{c_1}(y_n,v), & \quad v \ge y_n, \rule[-15pt]{0pt}{15pt}\\
{c_1}(y_n,y_n), & \quad v \le y_n. 
\end{array} \right.
\end{equation}
For the remainder of this proof, we simplify notation by setting 
\begin{equation}\label{eq-nu-nun}
Q(\cdot\,) := Q(\cdot\,;y,z), \quad \text{and}\quad  Q_n(\cdot\,) := Q(\cdot\,;y_n,z_n).
\end{equation}
Since $f$ is continuous, for every $t \in {\RR}$ and $\epsilon > 0$ the set  $\{v \in {\cal{E}}: f(v) > t+{\epsilon}\}$ is an open set.  Moreover, $c_1$ is uniformly continuous on any compact subset in $\ovl{\cal R}$.  Hence, for sufficiently large $n$, $v\in\{f>t+\epsilon\}$ implies $v\in\{f_n > t\}$. By Condition \ref{Q-cnds}(b), the measures $Q_n$ converge weakly to $Q$  on $\cal{E}$ and thus by the Portmanteau Theorem (cf. Theorem~3.3.1 on p.\,108 of \cite{ethi:86}) for the first inequality below and the inclusion $\{f>t+\epsilon\}\subset \{f_n>t\}$ for $n$ sufficiently large for the second inequality,
\begin{equation}                 
Q(\{ f > t+\epsilon \}) \leq \liminf_{n\rightarrow \infty} Q_n(\{ f> t+\epsilon \}) \leq \liminf_{n\rightarrow \infty} Q_n(\{f_n> t \}).
\end{equation}
Since $\epsilon$ is arbitrary, the hypothesis of  Lemma~2.1 in \cite{serf:82} is satisfied and it therefore follows that 
$$\int{f(v)Q(dv)}  \le \liminf_{n\rightarrow \infty}\int{f_n(v)Q_n(dv)}.$$
By \cndref{Q-cnds}(a) and the notation \eqref{eq-nu-nun}, $Q(\cdot)$ has its support in $(y,z]$ and similarly for $Q_n(\cdot)$.  Therefore 
$$\wdh{c}_1(y,z) = \int{f(v)Q(dv)}  \qquad \mbox{and} \qquad \wdh{c}_1(y_n,z_n) = \int{f_n(v)Q_n(dv)}$$
implying that \eqref{eq-c1hat-lsc} holds true.
\end{proof}



\subsection{Analysis of nominal $(s,S)$ Ordering Policies}\label{sect-sS-policy}
Both this paper and \cite{helm:18} rely on characterizing the long-term average cost for $(s,S)$-ordering policies in the cases of deficient supplies or of full supplies using a renewal reward theorem.  For $(y,z) \in {\cal R}$, define the nominal $(y,z)$-ordering policy $(\tau,Z)$ such that $\tau_0=0$ and
\begin{equation} \label{sS-tau-def}
\tau_k = \inf\{t > \tau_{k-1}: X(t-) \leq y\}, \quad \mbox{ and } \quad Z_k = z,
\quad  k \geq 1,
\end{equation}
in which $X$ is the inventory level process satisfying \eqref{controlled-dyn} with this ordering policy. The above definition of $\tau_k$ must be slightly modified when $k=1$ to be $\tau_1 = \inf\{t \geq 0: X(t-) \leq y\}$ to allow for the first jump to occur at time $0$ when $x_0 \leq y$.  Observe that $X$ is a delayed renewal process since the single distribution $Q(\cdot\,;,y,z)$ is used to determine the random supply for all orders $k \geq 2$; it is a renewal process when $y \leq x_0$.  We note that the definition of $\tau_k$ in \eqref{sS-tau-def} needs to be more precisely stated as in Section 6 of \cite{helm:19} due to the particular construction of the mathematical model.  However, the definition in \eqref{sS-tau-def} provides the correct intuition so we rely on this simpler statement of the intervention times.

Theorem~2.1 of \cite{sigm:93} provides existence and uniqueness of the stationary distribution for the process $X$ arising from a nominal $(y,z)$-ordering policy for any $(y,z) \in {\cal R}$ and moreover, the one-dimensional distributions $\PP(X(t) \in \cdot)$ converge weakly to the stationary distribution as $t$ tends to infinity.  A straightforward generalization of Proposition~3.1 of \cite{helm:15b} characterizes the density $\pi$ of the stationary distribution for $X$ and the long-run frequency $\hat\kappa = \frac{1}{\hat{B\zeta}(y,z)\rule{0pt}{10pt}}$ of orders.  

By renewal theory, the long-term average running cost 
for the nominal $(y,z)$-ordering policy, cf. \eqref{sS-tau-def}, equals:  
\begin{equation}\label{eq-c0-integral-slln}
\lim_{t\rightarrow \infty} \frac{1}{t} \int_{0}^{t} c_{0}(X(s))ds = \frac{\hat{Bg_0}(y,z)}{\hat{B\zeta}(y,z)} \;\;(a.s. \mbox{ and in } L^1),
\end{equation} 
and therefore the long-term average cost $J(\tau,Z)$ of \eqref{lta-obj-fn} is given by
\begin{equation}\label{eq-sS-cost}
J(\tau,Z) = \frac{\hat{c}_{1}(y,z) + \hat{Bg_{0}}(y,z)}{\hat{B\zeta}(y,z)}.
\end{equation}

Motivated by \eqref{eq-sS-cost}, define the function ${H}_0: \overline{\cal R} \rightarrow \overline{\RR^+}$ by
\begin{equation} \label{eq-Fhat-fn}
{H}_0(y,z) := \left\{\begin{array}{cl}\displaystyle 
\frac{\hat{c}_1(y,z) + \hat{Bg_0} (y,z)} {\hat{B\zeta}(y,z)}, & \quad (y,z) \in {\mathcal R}, \rule[-15pt]{0pt}{15pt}\\
\infty, & \quad (y,y) \in \overline{\cal R}.
\end{array} \right.
\end{equation}
$H_0$ is an adaptation of the function $F_0$ in \cite{helm:18} to the case of random yields. 
Recall that $Q(\,\cdot\,;y,z)$ has its support in $(y,z]$ and the collection is weakly convergent.  Since $g_0$ and $\zeta$ are continuous, it follows that $\wdh{Bg_0}$ and $\wdh{B\zeta}$ are also continuous, as well as being nonnegative.  Therefore ${H}_0$ is lower semicontinuous on $\overline{\mathcal{R}}$ due to \propref{c1hat-lsc}. 

Similar to the case of non-deficient deliveries, our goal is to minimize ${H}_0$.  Since $c_1 > 0$, and hence $\hat{c}_1$ is positive, ${H}_0(y,z) > 0$ for every $(y,z) \in \overline{\mathcal{R}}$. Thus, $\inf_{(y,z)\in \overline{\mathcal{R}}} {H}_0(y,z) =: {H}_0^* \geq 0$.  The models with a natural boundary allow ${H}_0^* = 0$ as a limit as the appropriate coordinate approaches the boundary point, in which case it immediately follows that there is no minimizing pair $(\yzstar,\zzstar)$ of ${H}_0$.  The imposition of \cndref{extra-cnd} below eliminates the possibility that $H_0^*=0$.

It is helpful to define a family $\{\mathfrak{P}(\cdot;y,z): (y,z)\in {\cal R}\}$ of probability measures on $\E$ as follows:
$$\mathfrak{P}(\Gamma;y,z) = \int_\Gamma B\zeta(y,v) \mbox{$\frac{1}{\wdh{B\zeta}(y,z)\rule{0pt}{10pt}}$}\, Q(dv;y,z), \qquad \Gamma \in \B(\E).$$
Note that the value $\mathfrak{P}(\Gamma;y,z)$ gives the proportion of the expected cycle length $\wdh{B\zeta}(y,z)$ due to the random effect distribution $Q(\,\cdot\,;y,z)$ delivering to inventory levels $v\in \Gamma$ following the order.  Also observe that $\mathfrak{P}(\,\cdot\,;y,z)$ inherits its support from $Q(\,\cdot\,;y,z)$.

The next result shows that the infimum $F_0^*$ of the function $F_0$ in \cite{helm:18}, see \eqref{eq-F-fn} below, is a lower bound for the value $ H_{0}^{*}$.  The function $F_0$ gives the long-term average cost of a $(y,z)$ policy for non-deficient supply models.

\begin{prop} \label{inf_comparison}
Assume Conditions \ref{diff-cnd} - \ref{cost-cnds}. Define the function 
\begin{equation} \label{eq-F-fn}
F_0(y,z) := \left\{\begin{array}{cl}\displaystyle 
\frac{c_1(y,z) + Bg_0 (y,z)} {B\zeta(y,z)}, & \quad (y,z) \in {\mathcal R}, \rule[-15pt]{0pt}{15pt}\\
\infty, & \quad { (y,z) \in \overline{\cal R} \text{ with } y=z,}
\end{array} \right.
\end{equation}
and let $F_0^* = \inf_{(y,z)\in \overline{\cal R}} F_0(y,z)$.  Then $
 H_{0}^{*}\ge F_{0}^{*}$.
\end{prop}
\begin{proof}
Observe that the function ${H}_0$ defined by \eqref{eq-Fhat-fn} can also be written as 
\begin{equation} \label{eq2-Fhat-fn}
{H}_0(y,z) := \left\{\begin{array}{cl}\displaystyle 
\int \frac{c_1(y,v) + Bg_0 (y,v)} {\wdh{B\zeta}(y,z)}\, Q(dv;y,z), & \quad (y,z) \in {\mathcal R}, \rule[-15pt]{0pt}{15pt}\\
\infty, & \quad (y,z) \in \overline{\cal R} \text{ with } y=z. \rule{0pt}{14pt}
\end{array} \right.
\end{equation}
Using the factor $\frac{B\zeta(y,v)}{B\zeta(y,v)}=1$, the expression for $H_0$ when $y < z$ yields
$$H_0(y,z) = \int \frac{c_1(y,v) + Bg_0(y,v)}{B\zeta(y,v)}\, \mathfrak{P}(dv;y,z) = \int F_0(y,v)\, \mathfrak{P}(dv;y,z) \geq F_0^*.$$
Taking the infimum over $(y,z)\in {\cal R}$ therefore establishes the result. 
\end{proof}

Similarly as in \cite{helm:18}, our main optimality result depends on the existence of a minimizing pair $(\yzstar,\zzstar)\in {\cal R}$ of ${H}_0$.  An important subtlety is that properties of the function ${H}_0$ on compact subsets of ${\cal R}$ and close to the boundary of ${\cal R}$ are not simply determined by the properties of the functions $c_1$, $g_0$ and $\zeta$ in these regions as they were for non-deficient supply models. Actually, the behavior of the function ${H}_0$ near the boundary crucially depends on properties of the measure-valued transition functions $Q(\cdot\,;y,z)$ as functions on ${\cal R}$ and, in particular, on the behavior of the function $\wdh{B\zeta}$ near the boundary. As a consequence, a proof of a general optimality result of an $(s,S)$-policy for inventory models with random supply requires additional conditions.  Before presenting these conditions, however, we identify an important relation between \cndref{Q-cnds}(c) and the family of measures $\{\mathfrak{P}(\,\cdot\,;y,z)\}$.

\begin{lem}\label{lem-non-compact}
Let $b$ be a natural boundary for which Condition \ref{Q-cnds} (c) holds.  Then for each interval $[d_1,d_2] \subset \I$ and for every $\check{z}$ with $d_2 < \check{z} < b$, 
\begin{equation} \label{non-compact}
\lim_{z\to b}\, \inf_{y\in [d_1,d_2]} \mathfrak{P}((\check{z},b);y,z) = 1.
\end{equation}
\end{lem}

 \begin{proof} Let $[d_1,d_2] $ and   $\check{z}$  be given as in the statement of the lemma. Denote $$ M: = \sup\{B\zeta(y,v): y \in  [d_1,d_2], v\in [y, \check z] \} < \infty.$$ Furthermore, let $\delta > 0$ be as in   Condition \ref{Q-cnds} (c). For any $\e > 0$, choose an $N \in \NN$ so that $N > \frac{2M}{\delta\,\e}$. 
 
Since $b$ is a natural boundary, $\lim_{v\to b}[\zeta(v) - \zeta(y)]  = \infty$ uniformly for   $y\in [d_1,d_2]$. Consequently,	for the $N\in \NN$ chosen above, there exists a $z_{N} < b$ (without loss of generality, we can assume that $z_{N} > \check z$) so that \begin{displaymath}
\zeta(v) - \zeta(y) \ge N, \quad \text{ for all } v \ge z_{N} \text{ and }y \in [d_1,d_2].
\end{displaymath}  Now, for the chosen $z_{N}$,  Condition \ref{Q-cnds} (c) says that we can find a $z_{\e} \in (z_{N}, b)$  so that \begin{displaymath}
Q([z_{N}, z]; y,z)  \ge \frac\delta{2}, \quad \text{ for all }z > z_{\e} \text{ and } y \in  [d_1,d_2].
\end{displaymath} 
Then for all $y \in  [d_1,d_2]$ and $z > z_{\e}$, we have \begin{align*} 
  \wdh{B\zeta}(y,z) &  = \int_{y}^{z_{N}} B\zeta(y,v) Q(dv; y,z) + \int_{z_{N}}^{z} B\zeta(y,v) Q(dv; y,z)   \\
& \ge  0 +  N Q([ z_{N}, z]; y,z) 
\ge \mbox{$\frac{N \delta}{2}$}.
\end{align*}  
Consequently, it follows that   for any $y \in  [d_1,d_2]$   and $z > z_{\e}$, we have
 \begin{align*} 
  \mathfrak{P}((\check{z},b);y,z) & =\int_{\check z}^{z} B\zeta(y,v)  \frac{1}{\wdh{B\zeta} (y,z)} Q(dv; y,z)\\
& = \frac{\int_{y}^{z} B\zeta(y,v) Q(dv; y,z)-  \int_{y}^{\check z} B\zeta(y,v) Q(dv; y,z)}{\wdh{B\zeta} (y,z)}\\
 & \ge  1 - \frac{M}{\wdh{B\zeta} (y,z) } \ge  1 - \mbox{$\frac{2M}{N\delta}$} > 1-\e. 
\end{align*} 
This establishes \eqref{non-compact} and hence completes the proof.   
\end{proof}

\begin{rem}
\cndref{Q-cnds}(c) is stronger than the conclusion of this lemma.  To see this, assume $b$ is a natural boundary, let \cndref{diff-cnd} hold and let $\zeta$ be given by \eqref{g0-fn}.  We identify a family $\cal Q$ for which \eqref{non-compact} holds but \cndref{Q-cnds}(c) fails.  We focus on the subset of ${\cal R}$ for which $B\zeta > 1$.  For each such $(y,z)$, let $\breve{y}$ satisfy $\breve{y} > y$ with $B\zeta(y,\breve{y})=\frac{1}{2}$; also set $m_1 := \frac{1}{\sqrt{B\zeta(y,z)}}$ and $m_0:= 1-m_1$.  
Now consider the random supply measures for $(y,z)$ with $B\zeta(y,z) > 1$ given by  
$$Q(\,\cdot\,;y,z) = m_0 \delta_{\breve{y}}(\cdot) + m_1 \delta_{z}(\cdot).$$
Notice that 
$$\wdh{B\zeta}(y,z) = B\zeta(y,\breve{y}) m_0 + B\zeta(y,z) m_1 = \frac{m_0}{2} + \sqrt{B\zeta(y,z)}$$
so for fixed $y$, $\wdh{B\zeta}(y,z) \to \infty$ as $z\to b$.  This convergence then implies \eqref{non-compact} holds for the fixed $y$ and a simple argument extends this to a uniform convergence for $y\in [d_1,d_2]$.

Now for $y\in [d_1,d_2]$ and $(y,z)$ with $B\zeta(y,z) > 1$, for any $\check{z} > d_2$, $Q((\check{z},b);y,z) = \frac{1}{\sqrt{B\zeta(y,z)}} \to 0$ as $z\to b$.  Hence \cndref{Q-cnds}(c) fails.
\end{rem}

Now, combined with Conditions \ref{diff-cnd}, \ref{Q-cnds} and \ref{cost-cnds}, the following set of conditions will be sufficient to guarantee the existence of a minimizer of the function $H_0$ on $\cal{R}$.
 
\begin{cnd} \label{extra-cnd}  
The following conditions hold: 
\begin{description}
\item[(a)]
The boundary $a$ is regular; or exit; or $a$ is a natural boundary for which there exists some $ (y_1,z_1) \in \mathcal{R}$ such that $H_0(y_1,z_1) < c_0(a)$. 

\item[(b)] The boundary $b$ is entrance; or $b$ is natural for which there exists some $(y_2,z_2)\in \mathcal{R}$ such that  $H_0(y_2,z_2) < c_0(b)$. 
\end{description}
\end{cnd}

\begin{rem}
In comparing the random supply model of this paper with the non-deficient supply model of \cite{helm:18}, we observe that \cndref{diff-cnd} is the same in each paper and  \cndref{cost-cnds} of this paper is Condition~2.2 of our 2018 paper.  The present \cndref{Q-cnds} exists only in this paper.  Furthermore, \cndref{extra-cnd} corresponds to Condition~2.3 in the 2018 paper.  It uses $H_0$ in place of $F_0$ to account for random supplies and also removes a monotonicity requirement of $F_0$ near natural boundaries.  
\end{rem}
%
%


We now state our main existence result, which when combined with \thmref{G0-feasible}, establishes the optimality of a nominal $(s,S)$ ordering policy within the large class of admissible policies.  See \sectref{sect:examples} for examples which illustrate these results.

\begin{thm} \label{F-optimizers}
Assume Conditions \ref{diff-cnd} -- \ref{cost-cnds} and \ref{extra-cnd} hold.  Then  there exists a pair $(\yzstar,\zzstar) \in \mathcal R$ such that
\begin{equation} \label{eq-F-optimal}
{H}_0(\yzstar,\zzstar) = {H}_0^* = \inf\{{H}_0(y,z): (y,z) \in \overline{\mathcal R}\}.
\end{equation}
\end{thm}

\begin{proof} 
The proof consists of several parts corresponding to pieces of the boundary of ${\mathcal R}$, the type of boundary point, and the values of $c_0$ at $a$ and $b$. Since much of the analysis of each part of the proof is similar, we shall only spell out the details of the case that $a$ and $b$ are natural boundaries. When $a$ is attainable or $b$ is an entrance boundary, the boundary is included in $\E$ so the minimum of $H_0$ may be achieved using a boundary point.  The proofs of these cases follow a similar line of argument.  

\begin{figure}[h]
\centering
\begin{tikzpicture} [scale=0.8]
%
\draw[->] (-3.5,0) -- (3.65,0)  node[right] {$y$};
\draw[->] (0,-3.5) -- (0,4.25) node[above] {$v$};
%
 \draw[scale=0.5,domain=-6.85:6.99,smooth,variable=\x,blue] plot ({\x},{\x});
%
\node at(3.42,3.42)[above right]{$(b,b)$};
\node at(-3.45,-3.45)[below left]{$(a,a)$};
\node at (-3.45,3.5) [above left] {$(a,b)$};
%
\draw  [dashed,red] (-3.45,-3.5) -- (-3.45,3.5);
\node at (-3.45,0)[below left] {$a$};
\draw  [dashed,red] (-3.45,3.5) -- (3.45,3.5);
\draw (3.45,-0.05) -- (3.45,0.05);
\node at (3.45,0)[below] {$b$};
%
\path [fill=lightgray] (-3.45,-2) rectangle (-2.55,2);
\node at (-3,-0.5) {$E_{3}$};
\node at (-2.55,0) [above left] {$\wdt{y}$};
\draw (-2.55,-0.05) -- (-2.55,0.05);
%
\path [fill=gray] (-2.55,2.75) rectangle (2,3.5);
\node at (0,3.125) {$E_{5}$};
\draw (-0.05,2.75) -- (0.05,2.75);
\node at (0,2.75) [below left] {$\wdt{z}$};
%
\path [fill=green] (-2.25,-2.25) -- (-2.25,-1) -- (2,3) -- (2,2) -- (-2.25,-2.25);
\node at (0,0.25) [above] {$E_{6}$};
%
\path [fill=yellow] (-3.45,2) rectangle (-2.55,3.5);
\node at (-3,2.75) {$E_4$};
%
\path [fill=cyan] (2,2)-- (3.5,3.5)-- (2,3.5)--(2,2);
\node at (2.15,3)[right]{$E_{1}$};
\draw[-](2,-.05)--(2,0.05); 
\node at (2,0)[below] {$z_\vareps$};
%
\path [fill=pink] (-3.45,-3.45)-- (-2,-2)--(-3.45,-2)--(-3.45,-3.45);
\node at (-3.35,-2.43)[right]{$E_{2}$};
\draw[-](-2,-.05)--(-2,0.05); 
\node at (-2,0)[above right] {$y_{\vareps}$};
\end{tikzpicture}
\caption{Neighborhoods of the Boundary} \label{bdry-nbhds}
\end{figure}
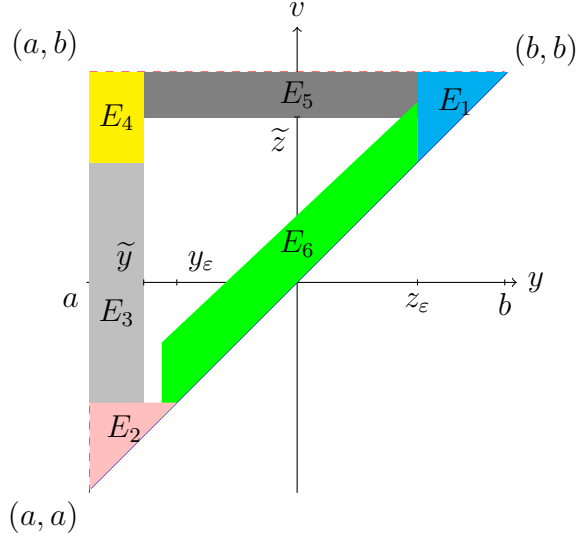

Our method of proof shows that $H_0$ is strictly greater than its infimum in a neighborhood of the boundary.  To begin, recall that
\begin{equation} \label{H0-expression} 
H_0(y,z) = \int_\E F_0(y,v)\, \mathfrak{P}(dv;y,z).
\end{equation}
The challenge is that $\mathfrak{P}(\,\cdot\,;y,z)$ may place mass throughout most of the interval $(y,z]$ so we need to be careful in developing the lower bounds of the integrand near different segments of the boundary; \figref{bdry-nbhds} aids in visualizing this analysis.  With reference to \figref{bdry-nbhds}, the bound 
\begin{equation} \label{F0-bound-1}
F_0(y,v) = \frac{c_1(y,v)+Bg_0(y,v)}{B\zeta(y,v)} > \frac{Bg_0(y,v)}{B\zeta(y,v)}
\end{equation}
will be used in the regions $E_1$, $E_2$, $E_3$, $E_4$ and $E_5$ while 
\begin{equation} \label{F0-bound-2}
F_0(y,v) = \frac{c_1(y,v)+Bg_0(y,v)}{B\zeta(y,v)} \geq \frac{c_1(y,v)}{B\zeta(y,v)} \geq \frac{k_1}{B\zeta(y,v)}
\end{equation} 
 will be  used for region $E_6$.

The two parts of \cndref{extra-cnd} can be combined to have a single pair $(y_1,z_1)\in {\cal R}$ for which $c_0(a)  \wedge c_0(b)> H_0(y_1,z_1)$.  Select $\vareps \in (0,1)$ so that
\begin{equation} \label{eps-choice}
c_0(a)\wedge c_0(b) > \frac{1+\vareps}{1-\vareps} H_0(y_1,z_1) + \vareps \quad \mbox{and}\quad \vareps < \frac{k_1}{H_0(y_1,z_1)}.
\end{equation}


\begin{itemize}
\item By \eqref{double-g0-zeta-diff-at-a} of \lemref{g0-at-a}, there exists some $z_\vareps$ such that 
$$\frac{Bg_0(y,v)}{B\zeta(y,v)} > H_0(y_1,z_1),\quad \forall z_\vareps \leq y < v < b.$$  
Define the neighborhood of $(b,b)$ to be $E_1 = \{(y,z)\in {\cal R}: z_\vareps \leq y < z < b\}$.  

\item Again by \eqref{double-g0-zeta-diff-at-a} of \lemref{g0-at-a}, there exists some $y_\vareps$ such that 
$$\frac{Bg_0(y,v)}{B\zeta(y,v)} > H_0(y_1,z_1), \quad \forall a < y < v \leq y_\vareps.$$  
Define the neighborhood of $(a,a)$ to be $E_2 = \{(y,z)\in {\cal R}: a < y < z \leq y_\vareps\}$.

\item Recall $x_0$ is the initial position.  Using $x_0$ as the fixed value in the two asymptotic results in \eqref{z-fixed-g0-zeta-diff-at-a} of \lemref{g0-at-a}, there exists $\ovl{y}$ and $\ovl{z}$ such that for $y \leq \ovl{y}$ and $v \geq \ovl{z}$, respectively,
\begin{equation} \label{E4-est}
\frac{Bg_0(y,x_0)}{B\zeta(y,x_0)} > H_0(y_1,z_1)  \quad \text{and}\quad \frac{Bg_0(x_0,v)}{B\zeta(x_0,v)} > H_0(y_1,z_1).
\end{equation}
For notational simplicity, we may assume $\ovl{y}=y_\vareps$ and $\ovl{z}=z_\vareps$ by using $y_\vareps\wedge \ovl{y}$ and $z_\vareps \vee \ovl{z}$ in the two previous parts as well as here.  Now define 
$$M:= \max_{y_\vareps \leq v \leq z_\vareps} (|g_0(v)| \vee |\zeta(v)|)$$ 
and note that $M < \infty$ since $g_0$ and $\zeta$ are continuous.  Using the fact that $\lim_{y\to a} \zeta(y) = -\infty$ along with \eqref{g0-zeta-ratio-at-a} of \lemref{g0-at-a}, there exists a $\wdt{y} \leq y_\vareps$ such that for $y \leq \wdt{y}$,
$$\frac{M}{\zeta(y)} \leq \vareps \quad \text{and}\quad \frac{g_0(y)}{\zeta(y)} > \frac{1+\vareps}{1-\vareps}H_0(y_1,z_1) + \vareps.$$
Define a neighborhood of the left boundary segment between $(a,y_\vareps)$ and $(a,z_\vareps)$ to be $E_3 = \{(y,z)\in {\cal R}: y \leq \wdt{y} \text{ and } y_\vareps \leq z \leq z_\vareps\}$.  Observe that for all $(y,z) \in E_3$,
\begin{align*}
\frac{Bg_0(y,v)}{B\zeta(y,v)} = \frac{g_0(y) - g_0(v)}{\zeta(y)-\zeta(v)} & \geq \frac{g_0(y) - M}{\zeta(y) + M} = \frac{\frac{g_0(y)}{\zeta(y)} - \frac{M}{\zeta(y)}}{1 + \frac{M}{\zeta(y)}} \\
& > \frac{\frac{1+\vareps}{1-\vareps}H_0(y_1,z_1) + \vareps - \vareps}{1+\vareps} = \frac{H_0(y_1,z_1)}{1-\vareps} > H_0(y_1,z_1).
\end{align*}

\item Again, let $z_\vareps$ be as in the definition of $E_1$, $y_\vareps$ be from $E_2$ and $\wdt{y}$ be as in $E_3$.  A key observation is that the inequalities \eqref{E4-est} establish that for $a < y \leq y_\vareps$ and $z_\vareps \leq v < b$, 
\begin{align*}
Bg_0(y,v) = Bg_0(y,x_0) + Bg_0(x_0,v) & > H_0(y_1,z_1) (B\zeta(y,x_0) + B\zeta(x_0,v)) \\
& = H_0(y_1,z_1) B\zeta(y,v)
\end{align*}
and therefore 
\begin{equation} \label{E3-est}
\frac{Bg_0(y,v)}{B\zeta(y,v)} > H_0(y_1,z_1), \quad \forall a < y \leq y_\vareps \text{ and } z_\vareps \leq v < b.
\end{equation}  
Since $\wdt{y} \leq y_\vareps$, this inequality holds in the neighborhood of $(a,b)$ defined by $E_4 := \{(y,z)\in {\cal R}: a < y \leq \wdt{y} \text{ and } z_\vareps \leq z < b\}$.  

\item Yet again, let  $z_\vareps$ be as in the definition of $E_1$ and $\wdt{y}$ be from $E_3$.   Now set $M_1 = \max_{\wdt{y}\leq v\leq z_\vareps} (g_0(v)\vee \zeta(v))$, noting that $M_1 \geq M$ since $[\wdt{y},z_\vareps] \supset [y_\vareps,z_\vareps]$.  Since $b$ is a natural boundary, $\lim_{v\to b} \zeta(v) = \infty$ and the asymptotic relation in \eqref{g0-zeta-ratio-at-a} of \lemref{g0-at-a} holds.  Thus there exists some $\check{z} \geq z_\vareps$ such that for $v \geq \check{z}$,
$$\frac{M_1}{\zeta(v)} \leq \vareps \quad \text{and}\quad \frac{g_0(v)}{\zeta(v)} > \frac{1+\vareps}{1-\vareps}H_0(y_1,z_1) + \vareps $$
and hence 
\begin{align} \label{est1} \nonumber
\frac{Bg_0(y,v)}{B\zeta(y,v)} = \frac{g_0(z) - g_0(y)}{\zeta(z)-\zeta(y)} & \geq \frac{g_0(v) - M_1}{\zeta(v) + M_1} = \frac{\frac{g_0(z)}{\zeta(z)} - \frac{M_1}{\zeta(z)}}{1 + \frac{M_1}{\zeta(z)}} \\
& > \frac{\frac{1+\vareps}{1-\vareps}H_0(y_1,z_1) + \vareps-\vareps}{1+\vareps} = \frac{H_0(y_1,z_1)}{1-\vareps}.
\end{align}
Using this $\check{z}$ in \eqref{non-compact} of \lemref{lem-non-compact}, there is some $\wdt{z} > \check{z}$ such that for $z > \wdt{z}$,
\begin{equation} \label{est2}
\inf_{y\in [y_\vareps,z_\vareps]} \mathfrak{P}((\check{z},b);y,z) > 1-\vareps.
\end{equation}
Define a neighborhood of the top boundary segment between $(y_\vareps,b)$ and $(z_\vareps,b)$ to be $E_5 = \{(y,z)\in {\cal R}: \wdt{y} \leq y \leq z_\vareps \text{ and } z \geq \wdt{z}\}$.   

\item Let $y_\vareps$, $z_\vareps$, $E_1$ and $E_2$ be as in the previous steps. From the first two analyzes, we know that for all $(y,z)\in E_1\cup E_2$, $H_0(y,z) \geq H_0(y_1,z_1)$.  We therefore only need to consider a neighborhood of the diagonal segment having $y\in [y_\vareps,z_\vareps]$.  Pick $\check{y}$ with $a < \check{y} < y_\vareps$ to allow a slight overlap with region $E_2$.

Since $\zeta$ is continuous, it is uniformly continuous on the interval $[\check{y},z_\vareps]$.  Let $\delta$ be such that $\check{y} \leq y \leq z_\vareps$ and $y \leq z \leq y+\delta$ implies $B\zeta(y,z) < \vareps$.  Define the neighborhood of the cropped diagonal to be $E_6 = \{(y,z)\in {\cal R}: \check{y} \leq y \leq z_\vareps, y < z \leq y+\delta \}$.  Recalling from \eqref{eps-choice} that $\vareps < \frac{k_1}{H_0(y_1,z_1)}$, it therefore follows from \eqref{F0-bound-2} that for all $(y,z)\in E_6$ and $y <v \leq z$ , 
$$F_0(y,v) > \frac{k_1}{B\zeta(y,v)} > \frac{k_1}{\vareps} > H_0(y_1,z_1).$$
\end{itemize}

Returning to \eqref{H0-expression}, observe that the integration is with respect to the second variable $v$ so is integration over the vertical line segment from the point $(y,y)$ on the diagonal to $(y,z)$.  In particular, for $(y,z) \in E_1\cup E_2 \cup E_3 \cup E_4 \cup E_6$, supp$(\mathfrak{P}(\,\cdot\,;y,z))$ is contained in this union.  

Now in the regions $E_1$ to $E_4$, combine \eqref{F0-bound-1} with the fact that $\frac{Bg_0(y,v)}{B\zeta(y,v)} > H_0(y_1,z_1)$ to see that $F_0(y,v) > H_0(y_1,z_1)$.  Similarly for region $E_6$, use the relation $F_0(y,v) > H_0(y_1,z_1)$ for the same result.  It now follows from \eqref{H0-expression} and the fact that $\mathfrak{P}(\,\cdot\,;y,z)$ is a probability measure that on regions $E_1$, $E_2$, $E_3$, $E_4$ and $E_6$, $H_0(y,z) > H_0(y_1,z_1)$ and hence the infimum does not occur in these regions or in the limit at the outer boundaries.

More care must be taken in region $E_5$ since for $(y,z) \in E_5$, supp$(\mathfrak{P}(\,\cdot\,;y,z))$ may not be contained in $\cup_{i=1}^6 E_i$ where $F_0(y,v)$ is larger than $H_0(y,v)$.  Using \eqref{H0-expression}, \eqref{F0-bound-1}, \eqref{est1} and \eqref{est2}, for $(y,z) \in E_5$,
\begin{align*}
H_0(y,z) = \int_\E F_0(y,v)\, \mathfrak{P}(dv;y,z) & \geq \int_{(\check{z},b)} F_0(y,v)\, \mathfrak{P}(dv;y,z) \\
& > \frac{H_0(y_1,z_1)}{1-\vareps}\, \mathfrak{P}((\check{z},b);y,z) > H_0(y_1,z_1).
\end{align*}
It thus follows that the infimum $H_0^*$ is not achieved or approached in $\cup_{i=1}^6 E_i$.  Therefore $H_0^*$ is achieved at some $(y_0^*,z_0^*)\in \ovl{(\cup_{i=1}^6 E_i)^c}\subsetneq {\cal R}$ since $H_0$ is lower semicontinuous on this compact region.
\end{proof}

\begin{rem}
For 
inventory models with non-deficient supply and   specially structured diffusion dynamics under appropriate conditions for the cost  functions, 
 the first order optimality conditions (see (3.17) of \cite{helm:15b}) involving $F_0$ of \eqref{eq-F-fn} can be utilized to obtain uniqueness of the optimizing policy.  The inclusion of the random yield measure adversely affects this analytical approach and we have been unable to derive general uniqueness results.
\end{rem}

\begin{rem}
Though the statement of \thmref{F-optimizers} requires \cndref{Q-cnds}, a careful examination of the proof reveals that only \eqref{non-compact} is used, which is implied by \cndref{Q-cnds}(c).  Thus existence of an optimizer holds when the weaker condition is imposed.  In addition, compared with Theorem~2.1 of \cite{helm:18}, our more careful analysis of $H_0$ at the boundaries using \eqref{eps-choice} proves the existence of an optimizing pair without the need of the monotonicity requirement of $F_0$ from Condition~2.3 of \cite{helm:18}.
\end{rem}

\section{Expected Occupation and Ordering Measures} \label{sect:occ-meas}
To establish general optimality of the $(\yzstar,\zzstar)$-policy of an inventory problem with random yield, we apply weak convergence arguments with average expected occupation  and average expected nominal ordering measures as well as expected stock-level measures which we now define.  For $(\tau,Z) \in \A$, let $X$ denote the resulting inventory level process satisfying \eqref{controlled-dyn}.  For each $t > 0$, define the average expected occupation measure ${\mu}_{0,t}$ on ${\cal E}$, and the average expected nominal ordering measure ${\nu}_{1,t}$ and stock-level measure ${\mu}_{1,t}$  on $\overline{\cal R}$ of the inventory process with random yield during the time interval $[0,t]$ by
\begin{equation} \label{mus-t-def}
\begin{array}{rcll}
 {\mu}_{0,t}(\Gamma_0) &:=& \displaystyle \mbox{$\frac{1}{t}$} \EE\left[\int_0^t I_{\Gamma_0}(X(s))\, ds\right], & \hspace{-0.5cm}\quad \Gamma_0 \in {\cal B}({\cal E}), \rule[-15pt]{0pt}{15pt} \\

 {\nu}_{1,t}(\Gamma_1)  &:=&  \displaystyle \mbox{$\frac{1}{t}$} \EE\left[\sum_{k=1}^\infty I_{\{\tau_k \leq t\}} I_{\Gamma_1}(X(\tau_k-),Z_k)\right], &\hspace{-0.5cm} \quad \Gamma_1 \in {\cal B}(\overline{\cal R}), \\

 {\mu}_{1,t}(\Gamma_2)  &:=& \displaystyle \mbox{$\frac{1}{t}$} \EE\left[\sum_{k=1}^\infty I_{\{\tau_k \leq t\}} {I_{\Gamma_2}}(X(\tau_k-),X(\tau_k))\right], & \hspace{-0.5cm} \quad \Gamma_2 \in {\cal B}(\overline{\cal R}). 
\end{array}
\end{equation}
The distinction between ${\nu}_{1,t}$ and ${\mu}_{1,t}$ is that the former is a measure on the (state,\,action) space while the latter is a measure on a (state,\,state) space, both spaces being correctly denoted by $\ovl{\cal R}$.
%

Using the construction of the underlying probability model of the inventory process $X$ corresponding to a policy $(\tau,Z) \in \A$ in \cite{helm:19}, we can rewrite the expected stock-level measure (up to time $t$) as follows: 
\begin{equation} 
\begin{array}{rcll}
{\mu}_{1,t}(\Gamma_2)  &=& \displaystyle \mbox{$\frac{1}{t}$} \sum_{k=1}^\infty \EE\left[ \EE\left[I_{\{\tau_k \leq t\}} I_{\Gamma_2}(X(\tau_k-),X(\tau_k)) | {\cal F}_{\tau_k-}\right] \right] & \\
&=& \displaystyle \mbox{$\frac{1}{t}$} \EE\left[\sum_{k=1}^\infty I_{\{\tau_k \leq t\}} \int I_{\Gamma_2}(X(\tau_k-),v)\, Q(dv; X(\tau_k-),Z_k)\right] & \\
&=& \displaystyle \mbox{$\frac{1}{t}$} \EE\left[\sum_{k=1}^\infty I_{\{\tau_k \leq t\}} \hat{I_{\Gamma_2}}(X(\tau_k-),Z_k)\right] \rule{0pt}{22pt} &\\
&=& \displaystyle \int \hat{I_{\Gamma_2}}(y,z)\, \nu_{1,t}(dy\times dz). & \\
\end{array}
\end{equation}
 Consequently, for any bounded, measurable $f$ and $t > 0$, we have 
\begin{eqnarray} \label{e:mu-nu-integral} \nonumber
\EE\Bigg[\sum_{k=1}^{\infty} I_{\{ \tau_{k}\le t\}} Bf (X(\tau_{k}-), X(\tau_{k})) \Bigg] &=& \int Bf(y,v) \mu_{1,t} (dy\times d v) \\
&=&  \int \wdh {Bf}(y,z) \nu_{1,t} (dy\times d z).
\end{eqnarray} 
Furthermore, using the measures $\mu_{0,t}, \mu_{1,t}$, and $\nu_{1,t}$, we can write for any $t > 0$, \begin{align} \label{e:cost-mu-nu-integral}
\nonumber t^{-1}  &\EE\bigg[\int_{0}^{t} c_{0}(X(s) ) ds + \sum_{k=1}^{\infty} I_{\{ \tau_{k}\le t\}} c_{1}(X(\tau_{k}-), X(\tau_{k})) \bigg] \\
\nonumber &  = \int c_{0}(x) \mu_{0,t}(dx) + \int c_{1}(y,v) \mu_{1,t} (dy\times d v) \\
& = \int c_{0}(x) \mu_{0,t}(dx) + \int \wdh c_{1}(y,z) \nu_{1,t} (dy\times d z).
\end{align} These observations will be used in Section \ref{sect:optim}.

Note, for the controlled process $X$, the expected stock-level measure ${\mu_{1,t}}$ counts the relative number of times the pairs of order-from-locations and inventory levels (after the supply has arrived) hit the set $\Gamma_2$ during the time interval $[0,t]$, while the expected {\em nominal}  ordering measure ${\nu_{1,t}}$ does so for the pairs of order-from-locations and control values (hitting the set $\Gamma_1$). 


Furthermore, if $a$ is a reflecting boundary and if $L_a$ denotes the local time of $X$ at $a$, define the average expected local time measure $\mu_{2,t}$ for each $t > 0$ to place a point mass on $\{a\}$ given by 
\begin{equation} \label{lta-local-time}
\mu_{2,t}(\{a\}) = \mbox{$\frac{1}{t}$} \EE[L_a(t)].
\end{equation}

\begin{rem} \label{masses-observation}
As in the case of inventory models with non-deficient yield in \cite {helm:18}, the average expected occupation measure ${\mu}_{0,t}$ is a probability measure on ${\cal E}$ for each $t > 0$.  In addition, for each $(\tau,Z) \in \A$ with $J(\tau,Z) < \infty$, ${\nu}_{1,t}$ has finite mass and $\limsup_{t\rightarrow \infty} {\nu}_{1,t}(\overline{\cal R}) \leq J(\tau,Z)/k_1$.  Observe that when $a$ is a sticky boundary, ${\mu}_{0,t}$ places a point mass at $a$ for those policies $(\tau,Z)$ that allow the process $X$ to stick at $a$ with positive probability.

Aside from the notation, the next two propositions and their proofs are the same as in Section 3 of \cite {helm:18}.  The two propositions focus on the relative compactness of the collection of $\mu_{0,t}$ measures and the associated convergence (or not) of the functionals with integrand $c_0$.
\end{rem}
 
\begin{prop}[Proposition~3.1 of \cite{helm:18}] \label{mu0-tightness}
Assume Conditions~\ref{diff-cnd} -- \ref{cost-cnds} are satisfied.  For $(\tau,Z) \in \A$, let $X$ denote the resulting inventory process satisfying \eqref{controlled-dyn}.  Let $\{t_i: i \in \NN\}$ be a sequence such that $\lim_{i\rightarrow \infty} t_i = \infty$ and for each $i$, define $\mu_{0,t_i}$ by \eqref{mus-t-def}.  If  $J(\tau,Z) < \infty$, then $\{\mu_{0,t_i}: i \in \NN\}$ is tight.
\end{prop}


\begin{prop}[Proposition~3.2 of \cite{helm:18}] \label{J-finite}
Assume Conditions~\ref{diff-cnd} -- \ref{cost-cnds} hold true.  Let $(\tau,Z) \in {\cal A}$ with $J(\tau,Z) < \infty$, $X$ satisfy \eqref{controlled-dyn}, and define $\mu_{0,t}$ by \eqref{mus-t-def} for each $t > 0$.  Then for each $\mu_0$ attained as a weak limit of some sequence $\{\mu_{0,t_j}\}$ as $t_j \rightarrow \infty$, 
$$\int_{\overline{\cal E}} c_0(x)\, \mu_0(dx) \leq J(\tau,Z) < \infty.$$
\end{prop}

We note that $c_0$ being infinite at a boundary implies that $\mu_0$ cannot assign any positive mass at this point.  In particular, for models in which $a$ is a sticky boundary and $c_0(a) = \infty$, any policy which allows $X$ to stick at $a$ on a set of positive probability incurs an infinite average expected cost for each $t$ and thus has $J(\tau,Z) = \infty$.  The requirement that $J(\tau,Z) < \infty$ therefore eliminates such $(\tau,Z)$ from consideration.

\section{The Auxiliary Function ${U}_0$}\label{sect-main result} 
To prove optimality of an $(s,S)$-policy for inventory models with random yield, we have to further adapt some of the concepts introduced in \cite {helm:18} to the case under consideration. In particular, we (slightly) modify the function $G_0=g_0 - F_0^*\zeta$ introduced in Section 4 of that paper. To this end, recall ${H}_0^*$ is the infimum of the function $H_0$ and \cndref{cost-cnds} requires continuity of $c_0$ at the boundary, even for finite, natural boundaries; $c_0$ may take value $\infty$ at the boundaries.
Define the auxiliary function ${U}_0$ on $\mathcal{E}$ by
\begin{equation} \label{G0-def}
U_0 := g_0 - H_0^* \zeta,
\end{equation}
and observe that the function $U_0$ differs from the function $G_0$ only as far as the constant $F_0^*$  is concerned; this constant is replaced by $H_0^*$.
Hence, the (new auxiliary) function ${U}_0$ inherits essential properties of the function  ${G}_0$.
Specifically, it is an element of $C(\mathcal{E})\cap C^2(\I)$, and it also extends uniquely to $\overline{\cal E}$ due to the existence of $(y_1,z_1)$ and $(y_2,z_2)$ in \cndref{extra-cnd} or $c_0$ being infinite at the boundaries.  
This observation follows immediately when $a$ is attainable and when $b$ is an entrance boundary since $\zeta$ is finite in these cases.  When $a$ or $b$ are natural boundaries, \lemref{g0-at-a} combined with \cndref{extra-cnd} shows that
\begin{equation} \label{nat-bdry-limits}
\lim_{x\rightarrow a} {U}_0(x) = \lim_{x\rightarrow a} (g_0(x) - {H}_0^* \zeta(x)) = \lim_{x\rightarrow a} \left(\mbox{$\frac{g_0(x)}{\zeta(x)}$} - {H}_0^*\right) \zeta(x) = -\infty
\end{equation}
and similarly $\lim_{x\rightarrow b} {U}_0(x) = \infty$.  
\color{black}
\begin{rem}
The function $U_0$ provides the following interpretation of the numerator of the function $H_0$.  Let $(y,z) \in \mathcal{R}$,  then
\begin{align*}
\hat{c}_1(y,z) + \hat{BU_0}(y,z) &= \hat{c}_1(y,z) + \hat{Bg_0}(y,z) - H_0^* \hat{B\zeta}(y,z) \\
&= \left(\frac{\hat{c}_1(y,z) + \hat{Bg_0}(y,z)}{\hat{B\zeta}(y,z)} - H_0^*\right) \hat{B\zeta}(y,z) = (H_0(y,z) - H_0^*) \hat{B\zeta}(y,z).
\end{align*}
Notice the relation $H_0^* \leq H_0(y,z)$ holds for all $(y,z) \in \mathcal{R}$. Thus, the function $\hat{c}_1(y,z) + \hat{BU_0}(y,z)$ gives the {\em increase in cost over a cycle}\/ incurred by using the nominal $(y,z)$-ordering policy rather than an optimal nominal ordering policy.
\end{rem}

Like the function $G_0$, the function $U_0$ also satisfies an (important)  system of relations.

\begin{prop} \label{qvi-ish}
Assume Conditions \ref{diff-cnd} -- \ref{cost-cnds} and \ref{extra-cnd} hold true.  Let $(y_0^*,z_0^*)\in {\cal R}$ be given by \thmref{F-optimizers} and let ${U}_0$ be as in \eqref{G0-def}.  Then ${U}_0$ is a solution of the system 
$$\left\{\begin{array}{rcll}
Af(x) + c_0(x) - {H}_0^*&=& 0, & \quad x\in \mathcal{I}, \\
\hat{Bf}(y,z) + \hat{c}_1(y,z) &\geq& 0,   & \quad (y,z) \in \overline{\mathcal{R}} \\
f(x_0) &=& 0,  & \\
\hat{Bf}(\yzstar,\zzstar) + \hat{c}_1(\yzstar,\zzstar) &=& 0. &
\end{array} \right.$$
Moreover, the first relation extends by continuity to $\overline{\cal E}$.
\end{prop} 

The proof is straightforward so is left to the reader.  With the appropriate use of the $\hat{~}$-operation in \eqref{hat-def}, the arguments of the proof of the following proposition are identical to those of Proposition 4.2 in \cite{helm:18} for models with non-deficient supply.  Similarly, Remark~4.2 of our 2018 paper remains valid, explaining the reason that the definitions of $g_0$ and $\zeta$ exclude the solutions to the homogeneous equations in \eqref{inhom-eq}.

\begin{prop} \label{sS-in-Q2}
Assume Conditions \ref{diff-cnd} -- \ref{cost-cnds} and \ref{extra-cnd}.  Let $x_0 \in \mathcal{I}$ be fixed.  For $a \leq y < z < b$, let $(\tau,Z)$ be the $(y,z)$-ordering policy defined by \eqref{sS-tau-def} and $X$ satisfy \eqref{controlled-dyn}.  Define the process $\wdt{M}$ by  
$$\wdt{M}(t) := \int_0^t \sigma(X(s)) {U}_0'(X(s))\, dW(s), \quad t \geq 0.$$
Then there exists a localizing sequence $\{\beta_n: n\in \NN\}$ of stopping times such that for each $n$, $\wdt{M}(\cdot \wedge \beta_n)$ is a martingale and the following transversality condition holds:
\begin{equation} \label{G0-transv-cnd}
\lim_{t\rightarrow \infty} \lim_{n\rightarrow \infty} \mbox{$\frac{1}{t}$} \EE[{U}_0(X(t\wedge \beta_n))] = 0.
\end{equation}
In addition, for a given $(y,z)$-policy, where $z$ denotes the {\em nominal} upper stock-level, defining $\mu_0^{(y,z)}$ to be the stationary measure of the controlled state process $X$ and $\mu_1^{(y,z)}$ to place point mass $\hat\kappa=\frac{1}{\hat{B\zeta}(y,z)\rule{0pt}{10pt}}$ (the long-run frequency of orders) on $\{(y,z)\}$, we have
$$ \int_{\cal E} A{U}_0(x)\, \mu_0^{(y,z)}(dx) + \hat{BU_0}(y,z)\, \hat\kappa = 0.$$
\end{prop}
\section{Policy Class ${\cal A}_0$ and Optimality} \label{sect:optim}

We prove optimality of an $(s,S)$-type policy in the class of admissible policies ${\cal A}$ for models with random yield very similarly as in Section~5 of \cite{helm:18} for models with non-deficient deliveries.  However, Proposition~5.3 and Corollary~5.6 of that paper require extensive modifications to apply to models with deficient supply.  These results and their proofs are carefully presented in this section.

\color{black}
Again, for models having a reflecting boundary point $a$, we are only able to prove optimality of a $(\yzstar,\zzstar)$-ordering policy within a slightly smaller class of admissible policies than the class ${\cal A}$ .  (Note there is no restriction on the class $\A$ when $a$ is not a reflecting boundary.)

\begin{defn} \label{admissible-class-A0-defn} 
For models in which $a$ is a reflecting boundary point, the class $\A_0 \subset \A$ consists of those policies $(\tau,Z)$ for which the transversality condition on the local-time process $L_a$ of the inventory process $X$,
\begin{equation} \label{no-local}
\lim_{t\rightarrow \infty} t^{-1} \EE[L_a(t)] = 0
\end{equation}
holds. 
\end{defn} 

The definition of an appropriate class of test functions ${\cal D}$ is as in \cite{helm:18}.
\begin{defn} \label{class-D-def}
A function $f$ is in ${\cal D}$ provided it satisfies
\begin{itemize}
\item[(a)] $f\in C(\overline{\cal E}) \cap C^2(\I)$ and there exists $L_f < \infty$ such that 
\begin{itemize}
\item[(i)] $|f| \leq L_f$;
\item[(ii)] $(\sigma f')^2 \leq L_f(1+c_0)$;
\item[(iii)] $|Af| \leq L_f$;
\end{itemize}
\item[(b)] 
\begin{itemize}
\item[(i)] for all models, at each boundary where $c_0$ is finite, $Af$ extends continuously to the boundary with a finite value;
\item[(ii)] when $a$ is a reflecting boundary, $|f'(a)| < \infty$; and 
\item[(iii)] when $a$ is a sticky boundary and $c_0(a) < \infty$, $\sigma f'$ extends continuously at $a$ to a finite value. 
\end{itemize} 
\end{itemize}
\end{defn}

Using the class ${\cal D}$ we have the following version of the limiting adjoint equation for inventory models with random supply.  

\begin{prop} \label{bdd-adjoint}
Assume Conditions \ref{diff-cnd} -- \ref{cost-cnds}.  Let $(\tau,Z) \in {\cal A}_0$ with $J(\tau,Z) < \infty$ and let $X$ satisfy \eqref{controlled-dyn}.  For $t > 0$, define $\mu_{0,t}$, $\mu_{1,t}$ and $\nu_{1,t}$ by \eqref{mus-t-def} and let $\mu_0$ be such that $\mu_{0,t_j} \Rightarrow \mu_0$ as $j\rightarrow \infty$ for some sequence $\{t_j: j \in \NN\}$ with $\lim_{j\rightarrow \infty} t_j = \infty$.  Then the limiting adjoint relation
\begin{equation} \label{bar}
\forall f \in {\cal D}, \quad \int_{\overline{\cal E}} Af(x)\, {\mu}_0(dx) + \lim_{j\rightarrow \infty} \int_{\overline{\cal R}} \hat{Bf}(y,z)\, {\nu}_{1,t_j}(dy\times dz) = 0
\end{equation}
holds.
\end{prop}

\begin{proof} Using the same arguments as those in the proof of Proposition 5.1 in \cite{helm:18}, we can derive 
\begin{displaymath}
 \int_{\overline{\cal E}} Af(x)\, {\mu}_0(dx) + \lim_{j\rightarrow \infty} \int_{\overline{\cal R}} {Bf}(y,v)\, \mu_{1,t_j}(dy\times dv) = 0.
\end{displaymath} 
Then \eqref{bar} follows from \eqref{e:mu-nu-integral}. 
\end{proof}

Using a similar proof as in Corollary~5.1 of \cite{helm:18} along with \propref{bdd-adjoint}, the existence of an optimal $(\yzstar,\zzstar)$ policy is obtained when $U_0 \in {\cal D}$.

\begin{cor}
Assume Conditions \ref{diff-cnd} -- \ref{cost-cnds} and \ref{extra-cnd}.  Suppose $U_0 \in {\cal D}$.  Then for every $(\tau,Z) \in {\cal A}_0$, $J(\tau,Z) \geq H_0^*$ and the $(\yzstar,\zzstar)$-ordering policy is optimal in the class ${\cal A}_0$, in which $(\yzstar,\zzstar)$ is given by \thmref{F-optimizers}.
\end{cor}

Unfortunately, it is frequently the case that $U_0 \notin \mathcal{D}$ so it is necessary to approximate $U_0$ by functions in $\mathcal{D}$ and pass to a limit.  Recall from \eqref{nat-bdry-limits} that when $a$ is a natural boundary, $U_0(a) := \lim_{x\rightarrow a} U_0(x) = -\infty$ and similarly, $U_0(b):= \lim_{x\rightarrow b} U_0(x) = \infty$ when $b$ is natural.
To proceed, we impose the following set of conditions.  

\begin{cnd} \label{AG0-unif-int}
Let $U_0$ be as defined in \eqref{G0-def}.   
\begin{itemize}
\item[(a)] There exists some $L < \infty$ and some $y_1 > a$ such that 
\begin{itemize}
\item[(i)] for models having $c_0(a) = \infty$,
$$\frac{c_0(x)}{(1 + |{U}_0(x)|)^2} + \frac{(\sigma(x) {U}_0'(x))^2}{(1 + |{U}_0(x)|)^3} 
\leq L, \qquad a < x < y_1;$$
\item[(ii)] for models in which $c_0(a) < \infty$, there is some $\epsilon \in (0,1)$ such that 
$$\frac{(\sigma(x) {U}_0'(x))^2}{(1 + |{U}_0(x)|)^{2+\epsilon}} \leq L, \qquad a \leq x < y_1.$$
\end{itemize}

\item[(b)] There exists some $L < \infty$ and some $z_1 < b$ such that 
\begin{itemize}
\item[(i)] for models having $c_0(b) = \infty$,
$$\frac{c_0(x)}{(1 + |{U}_0(x)|)^2} 
+ \frac{(\sigma(x){U}_0'(x))^2}{(1+|{U}_0(x)|)(1+c_0(x))} \leq L, \qquad z_1 < x < b;$$
\item[(ii)] for models in which $c_0(b) < \infty$, there is some $\epsilon \in (0,1)$ such that 
$$\frac{(\sigma(x) {U}_0'(x))^2}{(1 + |{U}_0(x)|)^{2+\epsilon}} + \frac{(\sigma(x){U}_0'(x))^2}{(1+|{U}_0(x)|)(1+c_0(x))} \leq L, \qquad z_1 < x \leq b.$$
\end{itemize} 
\item[(c)] \begin{itemize}
\item[(i)] When ${U}_0(a) > -\infty$, or when $a$ is a sticky boundary with $c_0(a) < \infty$,  $\displaystyle\lim_{x\rightarrow a} \sigma(x) {U}_0'(x)$ exists and is finite; 
\item[(ii)] when $a$ is a reflecting boundary, ${U}_0'(a)$ exists and is finite; and 
\item[(iii)] when ${U}_0(b) < \infty$,  $\displaystyle\lim_{x\rightarrow b} \sigma(x) {U}_0'(x)$ exists and is finite.
\end{itemize}
\end{itemize}
\end{cnd}

First note that the bound in \cndref{AG0-unif-int}(b,i) at the boundary $b$ is more restrictive than the similar bound in \cndref{AG0-unif-int}(a,i) at $a$ since
\begin{equation} \label{bi-implies-ai}
\frac{(\sigma(x) U_0'(x))^2}{(1 + |U_0(x)|)^3} = \frac{(\sigma(x)U_0'(x))^2}{(1+|U_0(x)|)(1+c_0(x))} \cdot \frac{1+c_0(x)}{(1+|U_0(x)|)^2} \leq L(1+L).
\end{equation}
The need for tighter restrictions at the boundary $b$ than at $a$ is not unexpected since there is no way to control the process from diffusing upwards whereas ordering can prevent the process from diffusing towards $a$.

The reason for having two different conditions in \cndref{AG0-unif-int}(a,b) based on whether $c_0$ at the boundary is finite or infinite is that any limiting measure $\mu_0$ of the collection $\{\mu_{0,t}\}$ arising from an admissible policy $(\tau,Z)$ having finite cost $J(\tau,Z)$ must place no $\mu_0$-mass at a boundary where $c_0$ is infinite.  A weak limit $\mu_0$ may have positive mass at a boundary when $c_0$ is finite.  Also notice the subtle assumption in \cndref{AG0-unif-int}(a,ii) and (b,ii) that the bounds extend to the boundary whereas there is no assumption needed at the boundary in \cndref{AG0-unif-int}(a,i) and (b,i).  

A sequence of functions $U_n \in {\cal D}$ which will approximate the auxilliary function $U_0$ will be defined using the function $h(x) = (-\frac{1}{8} x^4 + \frac{3}{4} x^2 + \frac{3}{8})I_{[-1,1]}(x) + |x|\, I_{[-1,1]^c}(x)$ defined in Section 5 of \cite{helm:18}. While the formal definitions of $U_n$ and $G_n$ are similar, there are striking differences between these two approximations when analyzing integrals of the form $\int_{\overline{\cal R}} \hat{BU_n}(y,z) \, {\nu}_{1,t_j}(dy\times dz)$ and $\int_{\overline{\cal R}} BG_n(y,z) \, \mu_{1,t_j}(dy\times dz)$, see the proof of Proposition \ref{BGn-c1-rel} below.
\color{black}

In the next lemma, we define  the sequence of functions $\{{U_n}: n\in \NN\} \subset {\cal D}$ which approximate $U_0$ and in the lemma following that one we examine the convergence of $A{U_n}$ and ${B{U_n}}$. 
\color{black}
\begin{lem}  \label{lem:G0-approx}
Assume Conditions \ref{diff-cnd} -- \ref{cost-cnds}, \ref{extra-cnd} and \ref{AG0-unif-int} with $U_0$ defined by \eqref{G0-def}.  For each $n \in \NN$, define the function $U_n$ by 
\begin{equation} \label{Gn-def}
U_n = \frac{U_0}{1+\frac{1}{n} h(U_0)}.
\end{equation}
Then $U_n \in {\cal D}$ and \begin{align*} 
 & \lim_{n\rightarrow \infty} AU_n(x) = AU_0(x), \quad \forall x \in {\cal I}, \\
& \lim_{n\rightarrow \infty} \wdh{BU_n}(y,z) =  \wdh{BU_0}(y,z), \quad \forall (y,z) \in \overline{\cal R}.
\end{align*}  Moreover, at each boundary where $c_0$ is finite, $\lim_{n\rightarrow \infty} AU_n \geq AU_0$.
\end{lem}
\begin{proof} 
That $U_n \in {\cal D}$ and the convergence of $AU_{n}$ can be proven using similar arguments as those in the proofs of 
  Lemmas 5.1 and 5.2 of \cite{helm:18}. Similarly, we can show that $\lim_{n\rightarrow \infty}{BU_n}(y,v) =  {BU_0}(y,v)$ for all $(y, v) \in \ovl{\mathcal R}$. This, together with the bounded convergence theorem, implies the desired convergence of $ \wdh{BU_n}(y,z)$ to $\wdh{B U_{0}}(y,z)$. 
\end{proof}


The following proposition gives the first important result involving $AU_n$ and $c_0$.

\begin{prop} \label{AGn-c0-rel}
Assume Conditions \ref{diff-cnd} -- \ref{cost-cnds}, \ref{extra-cnd} and \ref{AG0-unif-int} hold.  Let $(\tau,Z) \in {\cal A}_0$ with $J(\tau,Z) < \infty$, $X$ satisfy \eqref{controlled-dyn}, $ {\mu}_{0,t}$ be defined by \eqref{mus-t-def} and let $ {\mu}_0$ be any weak limit of $\{{\mu}_{0,t}\}$ as $t\rightarrow \infty$.  Define $U_n$ by \eqref{Gn-def}.  Then 
$$\liminf_{n\rightarrow \infty} \int_{\overline{\cal E}} (AU_n(x) + c_0(x))\, {\mu}_0(dx) \geq \int_{\overline{\cal E}} (AU_0(x) + c_0(x))\, {\mu}_0(dx) \geq H_0^*.$$
\end{prop}

The proof uses \cndref{AG0-unif-int} and is again very similar to the proof of Proposition~5.2 for non-deficient supply models in \cite{helm:18}.  It is therefore left to the reader. 

We next establish a similar result involving $\wdh{BU_n}$ and $\wdh{c_1}$, though the lack of tightness of $\{ {\nu}_{1,t}\}$ means that the result cannot be expressed in terms of a limiting measure.

\begin{prop} \label{BGn-c1-rel}
Assume Conditions \ref{diff-cnd} - \ref{cost-cnds}, \ref{extra-cnd} and \ref{AG0-unif-int} hold.  Let $(\tau,Z) \in {\cal A}_0$ with $J(\tau,Z) < \infty$ and $X$ satisfy \eqref{controlled-dyn}.  Let $\{t_j:j\in \NN\}$ be a sequence such that $\lim_{j\rightarrow \infty} t_j= \infty$ and 
$$J(\tau,Z) = \lim_{j\rightarrow \infty} \mbox{$\frac{1}{t_j}$}\EE\Bigg[\int_0^{t_j} c_0(X(s))\, ds + \sum_{k=1}^\infty I_{\{\tau_k\leq t_j\}} c_1(X(\tau_k-),X(\tau_k))\Bigg].$$ 
For each $j$, define $\nu_{1,t_j}$ by \eqref{mus-t-def} and, with $U_0$ given in \eqref{G0-def}, define $U_n$ by \eqref{Gn-def}.  Then 
\begin{equation}  \label{Fatou-bdd} 
\liminf_{n\rightarrow \infty} \liminf_{j\rightarrow \infty} \int_{\overline{\cal R}} (\wdh{BU_n}(y,z) + \wdh{c_1}(y,z))\,  \nu_{1,t_j}(dy\times dz) \geq 0.
\end{equation}
\end{prop}

The proof of this  proposition is very long and technical. In a nutshell, the desired assertion \eqref{Fatou-bdd} follows from the progression of   Lemmas \ref{lem-U0-bdd}, \ref{lem-Rn1+Rn2}, \ref{lem-Rn2}, and \ref{lem-Rn1}. Let us briefly describe the idea here. First we observe in Lemma  \ref{lem-U0-bdd} that \eqref{Fatou-bdd} holds true if the function $U_{0}$ is uniformly bounded. Consequently, we only need to focus on the case when $U_{0}$ is unbounded, which, necessarily implies that  either $U_{0}(a) =-\infty$ or $U_{0}(b) =\infty$. We present only the case when $U_{0}(a) =-\infty$ and $U_{0}(b) =\infty$; the other cases (either  $U_{0}(a) >-\infty$ and $U_{0}(b) =\infty$, or  $U_{0}(a) =-\infty$ and $U_{0}(b) < \infty$) follow from similar arguments and are left to the reader.  Lemma \ref{lem-Rn1+Rn2} observes that the integrand $ \wdh{BU_n}(y,z) + \wdh{c_1}(y,z)$ of \eqref{Fatou-bdd} is bounded below by  the sum of two terms $\wdh R_{n,1} $ and $\wdh R_{n,2}$. Then we show in Lemmas  \ref{lem-Rn2} and \ref{lem-Rn1} that the double limits inferior involving  $\wdh R_{n,2} $ and $\wdh R_{n,1}$, respectively, are nonnegative, thus establishing \eqref{Fatou-bdd}. 

The analysis of each double limit inferior follows similar lines of reasoning, though significantly more effort is required for the term involving $\wdh R_{n,1} $. First $\ovl{\cal R}$ is partitioned into appropriate  subsets in the proofs of Lemmas  \ref{lem-Rn2} and \ref{lem-Rn1}. Detailed analyses reveal that the inner integrand  $\wdh R_{n,1} $  or $\wdh R_{n,2}$  is bounded below over these subsets of $\ovl{\mathcal R}$;  and taking limits leads to the desired result.  The limiting result for $\wdh R_{n,1}$ requires the ASC condition of \cndref{Q-cnds}(c) for the region $\Gamma_4$ in \figref{partition}.  For the subset $\Gamma_5$ of $\ovl{\mathcal R}$ in \figref{partition}, the analysis of the double limit inferior requires subtle  weak convergence analysis related to the measures $\{\nu_{1,t_j}\}$ as well.  

We now supply the details of the arguments.

\begin{lem}\label{lem-U0-bdd}
Let $U_0$ be defined by \eqref{G0-def}.  If $U_{0}$ is uniformly bounded, then \eqref{Fatou-bdd} holds.
\end{lem}

\begin{proof}
Suppose $\sup_{x\in \I} |U_{0}(x)| \le K$ for some positive constant $K\geq 1$. Recall the nonnegativity of $\wdh{BU}_0 + \wdh{c}_1$ from \propref{qvi-ish}.  Then 
\begin{align} \label{e5.6}
\nonumber & \int_{\overline{\cal R}} (\wdh{BU_n}(y,z) + \wdh{c_1}(y,z))\,   \nu_{1,t_j}(dy\times dz) \\
\nonumber & \ \ =\int_{\overline{\cal R}} (\wdh{BU_0}(y,z) + \wdh{c_1}(y,z))\,   \nu_{1,t_j}(dy\times dz)  + \int_{\overline{\cal R}} (\wdh{BU_n}(y,z) -\wdh{BU_0}(y,z))\,   \nu_{1,t_j}(dy\times dz)  \\
  &\ \  \ge \int_{\overline{\cal R}} (\wdh{BU_n}(y,z) -\wdh{BU_0}(y,z))\,   \nu_{1,t_j}(dy\times dz).
\end{align}
Now using the definition of $U_{n}(\cdot)$,  for any $(y,v) \in \ovl{\mathcal R}$ 
\begin{align*} 
 | B U_{n}(y,v) - BU_{0}(y,v)| & = \bigg|\frac{U_{0}(v)}{1+ \frac1n h(U_{0}(v))} - \frac{U_{0}(y)}{1+ \frac1n h(U_{0}(y))} - U_{0}(v) + U_{0}(y) \bigg|  \\
  & = \bigg|\frac{U_{0}(y) h(U_{0}(y))}{n(1+ \frac1n h(U_{0}(y)))}- \frac{U_{0}(v) h(U_{0}(v))}{n(1+ \frac1n h(U_{0}(v)))}\bigg|\\
  & \le \mbox{$\frac{2K^{2}}{n}$}.
\end{align*} 
As a result, for any $(y,z) \in \ovl{\mathcal R}$, we have 
\begin{align*} 
 \wdh{BU_n}(y,z) -\wdh{BU_0}(y,z)) & = \int_{y}^{z}[B U_{n}(y,v) - BU_{0}(y,v)] Q(dv; y,z)   \\
  & \ge - \int_{y}^{z} \mbox{$\frac{2K^{2}}{n}$}\, Q(dv; y,z) = - \mbox{$\frac{2K^{2}}{n}$}. 
\end{align*} 
Employing this lower bound in \eqref{e5.6} gives 
\begin{displaymath}
 \int_{\overline{\cal R}} (\wdh{BU_n}(y,z) + \wdh{c_1}(y,z))\, \nu_{1,t_j}(dy\times dz)  \ge  \int_{\overline{\cal R}} - \mbox{$\frac{2K^{2}}{n}$} \, \nu_{1,t_j}(dy\times dz)  = - \mbox{$\frac{2K^{2}}{n}$}\, \nu_{1,t_j}(\ovl{\cal R}).
\end{displaymath} 
The bound on the asymptotic limit of $\nu_{1,t_j}(\ovl{\cal R})$ as $j\to \infty$ in \remref{masses-observation} implies that
\begin{displaymath}
\liminf_{j\to\infty} \int_{\overline{\cal R}} (\wdh{BU_n}(y,z) + \wdh{c_1}(y,z))\, \nu_{1,t_j}(dy\times dz)  \geq - \mbox{$\frac{2K^2\, J(\tau,Y)}{n\, k_1}$}.
\end{displaymath}
Now letting $n\to \infty$ yields \eqref{Fatou-bdd}.
\end{proof}

For the remaining lemmas, assume $U_0$ is unbounded with $U_0(a)=-\infty$ and $U_0(b)=\infty$.  

\begin{lem}\label{lem-Rn1+Rn2}
Let $U_0$ be defined by \eqref{G0-def} and $U_n$ by \eqref{Gn-def}.  Then
\begin{equation}
\label{eq:hat-c1+BGn-est}
\begin{aligned}
(\hat{B U_{n}} + \hat c_{1})(y,z) & \geq \int_{y}^{z} R_{n,1}(y,v)  Q(dv; y,z) + \int_{y}^{z} R_{n,2}(y,v)  Q(dv; y,z)\\& = \hat R_{n,1}(y,z) +\hat R_{n,2}(y,z).
\end{aligned}
\end{equation}
in which 
\begin{align}
\label{eq:Rn1-defn} 
& R_{n,1}(y,v) := \frac{BU_{0}(y,v)+c_1(y,v)}{[1+\frac{1}{n} h( U_0(v))][1+\frac{1}{n} h(U_0(y))]},\\  \label{eq:Rn2-defn}
& R_{n,2}(y,v) := \frac{U_{0}( v) h(U_{0}(y)) - U_{0}(y) h(U_{0}(v))}{n [1+\frac{1}{n} h(U_0(v))][1+\frac{1}{n} h(U_0(y))]}. \rule{0pt}{20 pt}
\end{align}
\end{lem}

\begin{proof}
Since $c_1$ is strictly positive, observe that 
\begin{align}  \label{Rn-defs} \nonumber  
  c_1(y,v) + BU_{n}(y,v) &= c_1(y,v) + \frac{U_{0}(v)}{1+\frac{1}{n} h(U_{0}(v))} - \frac{U_{0}(y)}{1+\frac{1}{n} h(U_{0}(y))}  \\ \nonumber 
  &=\; \frac{BU_{0}(y,v)+c_1(y,v)}{[1+\frac{1}{n} h(U_{0}(v))][1+\frac{1}{n} h(U_{0}(y))]}  + \frac{U_{0}( v) h(U_{0}(y)) -   U_{0}(y) h(U_{0}(v))}{n [1+\frac{1}{n} h(U_{0}(v))][1+\frac{1}{n} h(U_{0}(y))]}\\ \nonumber 
  &  \ \ \ +\;c_1(y,v) \left(1 - \frac{1}{[1+\frac{1}{n} h(U_{0}(v))][1+\frac{1}{n} h(U_{0}(y))]}\right) \\ \nonumber 
	&\geq \frac{BU_{0}(y,v)+c_1(y,v)}{[1+\frac{1}{n} h( U_0(v))][1+\frac{1}{n} h(U_0(y))]} + \frac{U_{0}( v) h(U_{0}(y)) -   U_{0}(y) h(U_{0}(v))}{n [1+\frac{1}{n} h(U_{0}(v))][1+\frac{1}{n} h(U_{0}(y))]}\\
  &=  R_{n,1}(y,v) + R_{n,2}(y,v).
\end{align}
Now integrating with respect to $Q(\,\cdot\,;y,z)$ yields \eqref{eq:hat-c1+BGn-est}.
\end{proof}

We now demonstrate that the double limit inferior of $R_{n,2}$ is nonnegative.

\begin{lem}\label{lem-Rn2} 
Let $R_{n,2}$ be defined by \eqref{eq:Rn2-defn}.  Then  
\begin{equation}  \label{R-n2-Fatou-bdd} 
\liminf_{n\rightarrow \infty} \liminf_{j\rightarrow \infty} \int_{\overline{\mathcal R}} \hat R_{n,2}(y,z)\,  \nu_{1,t_j}(dy\times dz) \geq 0.
\end{equation}
\end{lem}

\begin{proof}

Since $U_{0}(a) = -\infty$, there exists some $y_1$, with $y_1 > a$ such that $U_{0}(x) < -1$ for all $x < y_1$.  Recall $h(x) = |x|$ on $(-\infty,-1)$ and $h(x) \geq |x|$ for all $x$.  Thus it follows that for all $(y,v)$ with $y \leq y_1$
\begin{equation}
\label{eq1:Rn>=0}
R_{n,2}(y,v) = \frac{|U_{0}(y)|(U_{0}(v) + h(U_{0}(v)))}{n [1+\frac{1}{n} h(U_{0}(v))][1+\frac{1}{n} |U_{0}(y)|]} \geq 0.
\end{equation}
Define $F_{1} : = \{(y,v)\in \lbar{\mathcal R}: a <  y \le y_{1}\}$.

Similarly, the condition $U_{0}(b) = \infty$ implies that there exists some $z_1$ with $ z_1 < b$ such that $U_{0}(v) \geq 1$ for $z_1 < v < b$. Thus for $(y,v) $ with $v > z_1$, 
\begin{equation}
\label{eq2:Rn>=0}
R_{n,2}(y,v) = \frac{U_{0}(v)(h(U_{0}(y)) - U_{0}(y))}{n [1+\frac{1}{n} h(U_{0}(v))][1+\frac{1}{n} h(U_{0}(y))]} \geq 0. 
\end{equation}
Set $F_{2} : = \{(y,v)\in \lbar{\mathcal R}: y_{1} <  y \le v,  z_{1} \le v  < b \}$ and also define the set $F_{3} : = \lbar{\mathcal R}\setminus (F_{1}\cup F_{2})$.  These sets are illustrated in \figref{F-regions}.

\begin{figure}[h]
\centering
\begin{tikzpicture} [scale=0.8]
\draw[->] (-3.5,0) -- (3.65,0)  node[right] {$y$};
\draw[->] (0,-3.5) -- (0,4.25) node[above] {$v$};
 \draw[scale=0.5,domain=-6.85:6.99,smooth,variable=\x,blue] plot ({\x},{\x});
\node at(3.42,3.42)[above right]{$(b,b)$};
\node at(-3.45,-3.45)[below left]{$(a,a)$};
\draw  [dashed,red] (-3.45,-3.5) -- (-3.45,3.5);
\node at (-3.45,0)[left] {$a$};
\draw  [dashed,red] (-3.45,3.5) -- (3.45,3.5);
\node at (0,3.5)[above left] {$b$};
\node at (0,2)[below right] {$z_{1}$};
\node at (-.88,1.4)[below]{$F_{3}$};
\node at (-2,0)[above right] {$y_{1}$};
\draw[dashed](-2.5,2.2) --(2,2.2);
\node at (0,2.2)[right] {$z_{1}$};
\draw[-](-2,-.05)--(-2,0.05); 
\path [fill=cyan] (-2,2) -- (2,2) -- (3.5,3.5)-- (-2,3.5) -- (-2,2);
\node at (0.4,2.7)[right]{$F_{2}$};
\path [fill=pink] (-3.45,-3.45)-- (-2,-2)--(-2,3.5)--(-3.45,3.5) -- (-3.45,-3.45);
\node at (-3,0.1)[right]{$F_{1}$};
\draw[-](-.05,2)--(0.05,2); 
\end{tikzpicture}
\caption{The regions $F_1$, $F_2$ and $F_3$} \label{F-regions}
\end{figure}
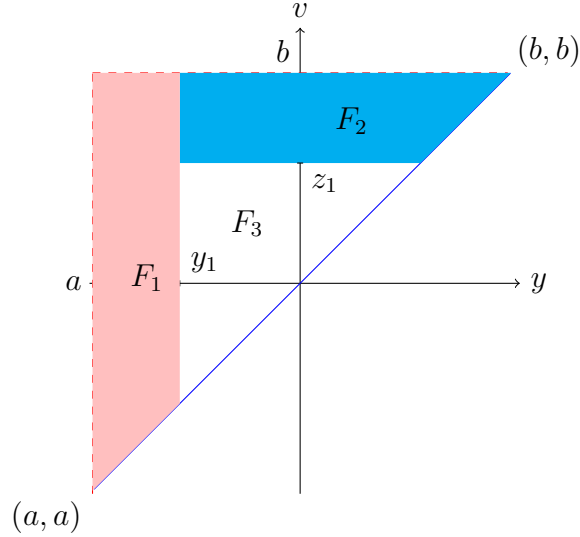

For $(y,z) \in F_{1} $, \eqref{eq1:Rn>=0} implies that 
\begin{align}\label{e1:Rn hat lower bd}
\hat R_{n,2}(y,z) & = \int_{y}^{z} R_{n,2}(y,v) Q(dv;y,z)  
  \ge 0.
\end{align}  
Establishing the result for regions $F_2$ and $F_3$ uses a common argument.  Considering the region $F_2$, the nonnegativity from \eqref{eq2:Rn>=0} implies that for $(y,z)\in F_{2}$,  
\begin{align*} 
 \hat R_{n,2}(y,z) & = \int_{y}^{z_{1}} R_{n,2}(y,v) Q(dv;y,z) +  \int_{z_{1}}^{z} R_{n,2}(y,v) Q(dv;y,z)       \\
  & \ge  \int_{y}^{z_{1}} R_{n,2}(y,v) Q(dv;y,z). 
\end{align*} 
For $(y,z)\in F_3$, $\hat R_{n,2}(y,z)  = \int_{y}^{z} R_{n,2}(y,v) Q(dv;y,z)$.  In each of these integrals, the upper limit of integration is bounded by $z_1$ so for each $(y,z)\in F_2\cup F_3$, we are only considering integrands $R_{n,2}$ on the closure of $F_3$.

Since the function $ U_{0}( v) h(U_{0}(y)) - U_{0}(y) h(U_{0}(v))$ is continuous, it is uniformly bounded on $\ovl{F}_3$. It follows that there exists some constant $K >0$ such that $|U_{0}(v) h(U_{0}(y)) -  U_{0}(y) h(U_{0}(v))| \le K$ and hence $|R_{n,2}(y,v) | \le \frac{K}{n}$. This, in turn, implies  that 
\begin{align}\label{e2-e3:Rn hat lower bd}
\hat R_{n,2}(y,z)   \ge  \int_{y}^{z\wedge z_{1}} R_{n,2}(y,v) Q(dv;y,z) \ge  - \frac{K}{n}\int_{y}^{z\wedge z_{1}} Q(dv;y,z) \ge - \frac{K}{n}. 
\end{align}

The inequalities \eqref{e1:Rn hat lower bd} and \eqref{e2-e3:Rn hat lower bd} imply that $\hat R_{n,2}(y,z) \ge - \frac{K}{n}$ for all $(y,z) \in \lbar{\mathcal R}$ and hence the asymptotic bound on the masses $\nu_{1,t_j}(\ovl{\cal R})$ in \remref{masses-observation} implies 
$$\liminf_{n\rightarrow \infty} \liminf_{j\rightarrow \infty} \int_{\overline{\mathcal R}} \hat R_{n,2}(y,z)\, \nu_{1,t_j}(dy\times dz) \geq 0.$$ 
\end{proof}

%
%

Turning to $R_{n,1}$, the proof of nonnegativity of the double limit inferior is more challenging.  

\begin{lem}\label{lem-Rn1} 
Let $R_{n,1}$ be given by \eqref{eq:Rn1-defn} and define $\wdh{R}_{n,1}$ by \eqref{eq:hat-c1+BGn-est}.  Then  
\begin{equation}  \label{R-n1-Fatou>=0} 
\liminf_{n\rightarrow \infty} \liminf_{j\rightarrow \infty} \int_{\lbar{\mathcal R}} \hat R_{n,1}(y,z)\, \nu_{1,t_j}(dy\times dz) \geq 0.
\end{equation}
\end{lem}

\begin{proof}

The argument begins with a similar line of reasoning as for \lemref{lem-Rn2} by establishing lower bounds on $R_{n,1}$ in various regions of $\ovl{\cal R}$.  \figref{partition} indicates the partition of $\ovl{\cal R}$ used in the proof.  The sets $\Gamma_1$ and $\Gamma_2$ are defined slightly differently depending on whether $a$ is attainable or natural and whether $b$ is entrance or natural.  When $a$ is attainable and $b$ is entrance, the partition can be slightly simplified.  In order that the proof apply to all types of boundary points, however, we adopt the same partition for every type of boundary.  

\begin{figure}[h]
\begin{center}
\begin{tikzpicture} [scale=0.8]
\draw[->] (-3.5,0) -- (3.65,0)  node[right] {$y$};
\draw[->] (0,-3.5) -- (0,4.25) node[above] {$v$};
 \draw[scale=0.5,domain=-6.85:6.99,smooth,variable=\x,blue] plot ({\x},{\x});
\node at(3.42,3.42)[above right]{$(b,b)$};
\node at(-3.45,-3.45)[below left]{$(a,a)$};
\draw  [dashed,red] (-3.45,-3.5) -- (-3.45,3.5);
\node at (-3.45,0)[left] {$a$};
\draw  [dashed,red] (-3.45,3.5) -- (3.45,3.5);
\node at (0,3.5)[above left] {$b$};
\path [fill=yellow] (-3.45,-2) rectangle (-2.5,3.5);
\node at (-2.89,1)[above]{$\Gamma_{3}$};
\node at (-2.5,0)[right] {$y_{0}$};
\draw[-](-2.5,-.05)--(-2.5,0.05); 
\node at (0,-2)[right] {$z_{0}$};
\draw[-](-.05,-2)--(0.05, -2); 
\path [fill=lightgray] (-2.5,2.8) rectangle (2,3.5);
\node at (-1.18,2.4)[above]{$\wdt\Gamma_{4}$};
\node at (-.88,1.4)[below]{$\Gamma_{5}$};
\node at (-.58,3.1)[right]{$\Gamma_{4}$};
\node at (2,0)[below] {$y_{1}$};
\draw[dashed](-2.5,2.2) --(2,2.2);
\draw[-](2.0,-.05)--(2.0,0.05); 
\node at (0,2.2)[right] {$\wdt{z}_{1}$};
\draw[-](-.05,2.2)--(0.05, 2.2); 
\node at (0,2.8)[right] {$z_{2}$};
\draw[-](-.05,2.8)--(0.05, 2.8); 
\path [fill=cyan] (2,2)-- (3.5,3.5)-- (2,3.5)--(2,2);
\node at (2.15,3)[right]{$\Gamma_{2}$};
\path [fill=pink] (-3.45,-3.45)-- (-2,-2)--(-3.45,-2)--(-3.45,-3.45);
\node at (-3.35,-2.43)[right]{$\Gamma_{1}$};
\end{tikzpicture}
\end{center}
\caption{Partition of $\ovl{\cal R}$} \label{partition}
\end{figure}
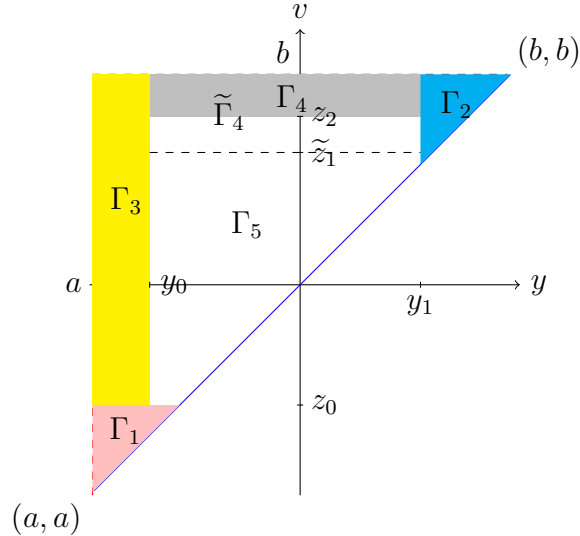

\begin{itemize}
\item[$\bullet$] When $a$ is attainable, $\zeta$ is bounded below on $\ovl{\E}$.  As a result, 
\begin{equation} \label{a-attainable-est}
\frac{BU_0(y,v) + c_1(y,v)}{B\zeta(y,v)} = \frac{Bg_0(y,v) + c_1(y,v)}{B\zeta(y,v)} - H_0^* \geq \frac{k_1}{B\zeta(y,v)} - H_0^*.
\end{equation}
Let $z_0$ satisfy $B\zeta(a,z_0) = \frac{k_1}{H_0^*}$ and define the set $\Gamma_1 = \{(y,z)\in \ovl{\cal R}: a \leq y \leq z < z_0\}$.  Then the monotonicity of $\zeta$ yields $0 < B\zeta(y,v) \leq \frac{k_1}{H_0^*}$ for $(y,v)\in \Gamma_1$ with $y < v$, and hence $BU_0(y,v) + c_1(y,v) \geq 0$, implying that $R_{n,1} \geq 0$ as well.  Continuity of $R_{n,1}$ up to the diagonal of $\Gamma_1$ then establishes $R_{n,1} \geq 0$ on $\Gamma_1$.

\item[$\bullet$] When $a$ is a natural boundary, \eqref{double-g0-zeta-diff-at-a} of \lemref{g0-at-a} with \cndref{extra-cnd}(a) implies that there is some $z_{0} > a$ so that $\frac{Bg_{0}(y,v)}{B\zeta(y,v)} \geq H_0(y_1,z_1) \ge H_{0}^{*} $ for all $ y \le v \le z_{0}.$  Define the region 
$$\Gamma_1 := \{(y,z)\in \ovl{\cal R}: a < y \leq z < z_0\}.$$  
As a result of the lower bound on the ratio, for $(y,v) \in \Gamma_1$, 
\begin{equation} \label{Rn1-numerator} 
0 \leq Bg_0(y,v) - H_0^* B\zeta(y,v) = BU_0(y,v) < BU_0(y,v) + c_1(y,v).
\end{equation}
Therefore from its definition, $R_{n,1} > 0$ on $\Gamma_1$.

\item[$\bullet$] When $b$ is an entrance boundary, $\zeta$ is bounded above on $\ovl{\E}$.  Set $y_1$ so that $B\zeta(y_1,b) = \frac{k_1}{H_0^*}$.  Define $\Gamma_2 = \{(y,z)\in \ovl{\cal R}: y_1 < y \leq z \leq b\}$.    Using the estimate in \eqref{a-attainable-est} and arguing similarly as for the boundary $a$, it follows that $R_{n,1} \geq 0$ on $\Gamma_2$.

\item[$\bullet$] When $b$ is a natural boundary, \eqref{double-g0-zeta-diff-at-a} of \lemref{g0-at-a} with \cndref{extra-cnd}(b) implies that there is some $y_{1} < b$ so that $\frac{Bg_{0}(y,v)}{B\zeta(y,v)} \ge H_{0}^{*} $ for all $y_{1}\le y \le v$.  Define the region 
$$\Gamma_2:=\{y,z)\in \ovl{\cal R}: y_1 < y \leq z < b\}.$$  
Then for $(y,v) \in \Gamma_2$, the relation \eqref{Rn1-numerator} again holds, implying that $R_{n,1}(y,v) > 0$.

\item[$\bullet$] Let $z_0$ be as in the definition of $\Gamma_1$.  Define $K_1 = \inf\{U_0(v): z_0 \leq v < b\}$ and observe that $K_1 > -\infty$.  Since $U_{0}(a) = -\infty$, continuity of $U_0$ at $a$ implies that there is some $y_0$ with $a < y_{0}  < y_{1}\wedge z_{0}$ such that $U_{0}(y) \le K_{1}$ for all $y < y_{0}$.  Define 
$$\Gamma_{3}:= \{(y,v)\in  \lbar{\mathcal R}: a < y< y_{0}, v \ge z_{0}\}.$$  
Then for all $(y,v) \in \Gamma_{3}$, 
\begin{align} \label{bound-on-E3} \nonumber
R_{n,1}(y,v) &= \frac{BU_{0}(y,v)+c_1(y,v)}{[1+\frac{1}{n} h(U_{0}(v))][1+\frac{1}{n} h(U_{0}(y))]} \\ 
&\ge \frac{K_{1}- K_{1} + c_1(y,v)}{[1+\frac{1}{n} h(U_{0}(v))][1+\frac{1}{n} h(U_{0}(y))]} > 0.
\end{align} 

\item[$\bullet$]  Following a similar argument, let $y_0$ and $y_1$ be as chosen above.  Define $K_2 = \sup\{|U_0(y)|: y_0 \leq y \leq y_1\}$.  Since $U_0(b) = \infty$, continuity implies  existence of some $\wdt{z}_1 < b$ for which $U_0(v) \geq K_2$ for all $v \geq \wdt{z}_1$.  Define the region
$$\wdt\Gamma_4 = \{(y,z)\in \ovl{\cal R}: y_0 \leq y \leq y_1, z \geq \wdt{z}_1\}.$$
For all $(y,v) \in \wdt\Gamma_4$, the numerator of $R_{n,1}$ has the bound $BU_0(y,v) + c_0(y,v) \geq K_2 - K_2 + c_0(y,v) > 0$ implying that $R_{n,1} > 0$ on $\wdt\Gamma_4$.
\end{itemize}

Turning briefly to $\wdh{R}_{n,1}(y,v) = \int_y^z R_{n,1}(y,v)\, Q(dv;y,z)$, notice that this is a line integral over the vertical segment $(y,y)$ to $(y,z)$.  For $\Gamma_1$, $\Gamma_2$ and $\Gamma_3$, these segments are entirely contained in the regions so it immediately follows that $\wdh{R}_{n,1} \geq 0$ on these regions.  For $(y,z)\in \wdt\Gamma_4$, the segment from $(y,y)$ to $(y,z)$ is not contained in $\wdt\Gamma_4$ and it is not necessary that $R_{n,1} \geq 0$ on the segment so a more careful analysis is required.

 Let $y_{0}, y_{1}, $ and $\wdt{z}_{1}$ be the values used to define the subsets $\Gamma_2$, $\Gamma_3$ and $\wdt\Gamma_4$.  Recall $K_{2} = \sup_{y_{0}\le y \le y_{1}}|U_{0}(y)|$.  Now set  
$$K_{3}: =  \sup_{y_{0}\le y \le y_{1}, y\le v \le \wdt{z}_{1}} |BU_{0}(y,v)+c_1(y,v)|. $$ 
Note that $| R_{n,1}(y,v) | \le K_{3} $ for all $n\in \mathbb N$ and $(y,v) \in \ovl{\mathcal R}$ with $y_{0}\le y \le y_{1}$ and $y\le v \le \wdt{z}_{1}$. In addition, observe that for any $(y,v) \in \wdt\Gamma_4$,  
\begin{align*} 
 R_{n,1}(y,v) & =  \frac{BU_{0}(y,v)+c_1(y,v)}{[1+\frac{1}{n} h(U_{0}(v))][1+\frac{1}{n} h(U_{0}(y))]}     \\
  & \ge \frac{ U_{0}(v) -  \sup_{y_{0}\le y \le y_{1}}|U_{0}(y)|  }{[1+\frac{1}{n} h(U_{0}(v))][1+\frac{1}{n} \cdot1\vee \sup_{y\in [y_{0}, y_{1}]} |U_{0}(y)|]} \\
  & = \frac{ U_{0}(v) -K_{2}}{[1+\frac{1}{n} h(U_{0}(v))][1+\frac{1}{n} \cdot1\vee K_{2}] }
  \\& =: f_{n}(v).
\end{align*} 
By the choice of $\wdt{z}_1$ and the definition of $K_2$, it is easy to see that for each $v \ge \wdt{z}_{1}$ fixed, $f_{n}(v)$ is increasing in $n$. Moreover, since $\lim_{v\to b} U_{0}(v) = \infty$, we have $\lim_{v\to b} f_{n}(v) = \frac{n}{1+\frac{1\vee K_{2}}{n}\rule{0pt}{10pt}}$ for each $n$. 

Using the interval $[y_0,y_1]$, let $\delta>0$ be given by \cndref{Q-cnds}(c).  We first fix an $N >(\frac{ 4K_{3}}{\delta}+1)\vee K_2$.  Since $\lim_{v\to b} f_{N}(v) = \frac{N}{1+\frac{1\vee K_{2}}{N}\rule{0pt}{10pt}}$, we can find a $z_1$ with $\wdt{z}_1 < z_1 < b$ so that $f_{N}(v) \ge \frac{N}{2} \ge \frac{2K_{3}}{\delta}$ for all $v \ge z_1.$ Consequently, for all $n \ge N$ and $(y,v)$ with $y_{0} \le y \le y_{1}$ and $v \geq   z_1$, we have 
\begin{displaymath}
R_{n,1}(y,v) \ge f_{n}(v) \ge f_{N}(v) \ge \frac{2K_{3}}{\delta}.
\end{displaymath} 
By the \cndref{Q-cnds}(c), there exists a $z_{2} > z_{1}$ so that 
$$\inf_{y\in [y_{0}, y_{1}]} Q((z_{1}, b);y,z) \ge \frac{\delta}{2},\quad \text{ for all }z> z_{2}.$$ 
Define $\Gamma_{4}: =\{(y,v) \in \ovl{\mathcal R}:    y_{0}\le y \le y_{1}  \text{ and }v > z_2\}$.  Recall, supp\,$Q(\,\cdot\,;y,z)\subset (y,z]$ so $Q((z_1,b);y,z) = Q((z_1,z];y,z)$.  Then for all  $n \ge N$ and all $(y,z) \in \Gamma_4$, 
\begin{align*} 
  \hat R_{n,1}(y,z) & 
  = \int_{(y,\wdt{z}_{1}]} R_{n,1}(y,v)\, Q(dv;y,z) +  \int_{(\wdt{z}_{1},z_{1}]}  R_{n,1}(y,v)\, Q(dv;y,z) \\
	& \qquad +  \int_{(z_{1},z] }  R_{n,1}(y,v)\, Q(dv;y,z)  \\ 
& \ge \int_{(y,\wdt{z}_{1}]}  (-K_{3})\, Q(dv;y,z) +   \int_{(\wdt{z}_{1},z_{1}]} 0\, Q(dv;y,z) +    \int_{(z_{1},z] } \frac{2K_{3}}{\delta}\, Q(dv;y,z)\\ 
 & \ge -K_{3} +  \frac{2K_{3}}{\delta}\cdot \frac{\delta}{2} =0.
\end{align*} 
Summarizing, on the set $\Gamma = \cup_{i=1}^4 \Gamma_i$, the function $\wdh{R}_{n,1} \geq 0$ so 
\begin{equation} \label{double-sets}
\liminf_{n\to\infty} \liminf_{j\to\infty} \int_{\Gamma} \wdh{R}_{n,1}(y,z)\, \nu_{1,t_j}(dy\times dz) \geq 0.
\end{equation}

Now define the set $\Gamma_5 = \ovl{\mathcal R} \backslash (\cup_{i=1}^4 \Gamma_{i})$; this compact set is displayed by the closure of the white region in \figref{partition}.  We need to show that  
\begin{align}\label{e:prop59E5-claim} 
\liminf_{n\to\infty}\liminf_{j\to\infty} & \int_{\Gamma_{5}} \wdh R_{n,1}(y,z) \nu_{1,t_{j}}(dy\times dz)  \ge 0.
\end{align} 

For each $n$, let $\{t_{j_k}\}\subset \{t_j\}$ be a subsequence such that 
$$\lim_{k\to\infty} \int_{\Gamma_{5}} \wdh R_{n,1}(y,z) \nu_{1,t_{j_k}}(dy\times dz) = \liminf_{j\to\infty} \int_{\Gamma_{5}} \wdh R_{n,1}(y,z) \nu_{1,t_{j}}(dy\times dz);$$
the dependence of the subsequence on $n$ is notationally suppressed.  Now restrict each $\nu_{1,t_{j_k}}$ to $\Gamma_5$ and observe that, trivially, the collection $\{\nu_{1,t_{j_k}}\}$ is tight and furthermore, $\nu_{1,t_{j_k}}(\Gamma_5) \leq \nu_{1,t_{j_k}}(\ovl{\cal R})$ for each $k$.  It therefore follows from \remref{masses-observation} that the masses $\{\nu_{1,t_{j_k}}(\Gamma_5)\}$ are uniformly bounded.  The properties of tightness and uniform boundedness imply that there exists some further subsequence $\{t_{j_{k_\ell}}\}$ and a measure $\ovl\nu_{1,n}$ on $\Gamma_5$ such that $\nu_{1,t_{j_{k_\ell}}} \Rightarrow \ovl\nu_{1,n}$ (see Theorem 8.6.2 of \cite{boga:07}); the dependence of the limiting measure on $n$ is now explicitly represented.  Note that since the measures are restricted to $\Gamma_5$, the weak convergence $\nu_{1,t_{j_{k_\ell}}} \Rightarrow \ovl\nu_{1,n}$ implies that 
$$\lim_{\ell \to\infty} \nu_{1,t_{j_{k_\ell}}}(\Gamma_5) = \lim_{\ell \to\infty} \int_{\Gamma_5} 1\, d\nu_{1,t_{j_{k_\ell}}} = \int_{\Gamma_5} 1\, d\ovl\nu_{1,n}(\Gamma_5) = \ovl\nu_{1,n}(\Gamma_5).$$

For each $n$, the function $\wdh R_{n,1}(y,z)$ can be shown to be lower semicontinuous by a similar argument as that for the proof of Proposition \ref{c1hat-lsc}.  In addition, $\wdh{R}_{n,1}$ inherits boundedness from the function $R_{n,1}$, which is continuous and uniformly bounded on the compact region $\Gamma_{5}$.  This bound is also uniform for all $n$ due to the definition of $R_{n,1}$.  Then applying Corollary~8.2.5 of \cite{boga:07},   
$$\liminf_{\ell\to\infty} \int_{\Gamma_5} \wdh{R}_{n,1}(y,z)\, \nu_{1,t_{j_{k_\ell}}}(dy\times dz) \geq \int_{\Gamma_5} \wdh{R}_{n,1}(y,z)\, \ovl\nu_{1,n}(dy\times dz).$$
The challenge in analyzing the right-hand side is the dependence on $n$ of both $\wdh{R}_{n,1}$ and $\ovl\nu_{1,n}$.  We will apply Lemma~2.1 in \cite{serf:82}, which concerns nonnegative functions.  Since $\wdh R_{n,1}$ is uniformly bounded on $\Gamma_{5}$ and over $n\in \NN$, there is a positive constant $R$ so that $\wdh R_{n,1}(y,z) + R \ge 0$ for all $(y,z) \in \Gamma_{5}$ and $n \in \NN$.  

Now let $\{n_m\}\subset \NN$ be a subsequence for which   
\begin{displaymath}
\lim_{m\to\infty} \int_{\Gamma_{5}} \wdh R_{n_m,1}(y,z)\,  \ovl \nu_{1,n_m} (dy\times dz) = \liminf_{n\to\infty}\int_{\Gamma_{5}} \wdh R_{n,1}(y,z)\, \ovl \nu_{1,n} (dy\times dz).
\end{displaymath}
The collection $\{\ovl \nu_{1,n_m}\}$, as measures on the compact set $\Gamma_{5}$, is tight and $\ovl\nu_{1,n_m}(\Gamma_5)$ inherits the uniform bound of \remref{masses-observation}.  Theorem~8.6.2 of \cite{boga:07} implies the existence of a further subsequence $\{\ovl \nu_{1,n_{m_i}}\}$ and a  measure $\ovl\nu$ so that $\ovl\nu_{1,n_{m_i}} \Rightarrow \ovl \nu$. 

We now verify the hypothesis of Lemma~2.1 of \cite{serf:82}.  Observe that Fatou's Lemma implies that for each $(y,z)\in \Gamma_5$, 
\begin{align} \label{e:R-n1-hat-liminf}
 \nonumber \liminf_{n\to\infty} \wdh R_{n,1}(y,z) & =  \liminf_{n\to\infty}  \int_{y}^{z} R_{n,1}(y,v) Q(dv; y, z)\\
\nonumber& \ge   \int_{y}^{z}  \liminf_{n\to\infty} R_{n,1}(y,v) Q(dv; y, z)  \\ 
\nonumber & = \int_{y}^{z} (BU_{0} (y,v) + c_{1}(y,v)) Q(dv; y, z) \\
& = \wdh{BU_{0}} (y, z) + \wdh c_{1}(y,z) \ge 0,
\end{align} 
where the last inequality follows from Proposition \ref{qvi-ish}.  Now briefly simplify notation by setting $f: =\wdh{BU_{0}} + \wdh c_{1}$.  Note that $f$ is nonnegative and lower semicontinuous on $ {\Gamma_{5}}$ by \propref{c1hat-lsc}.  Moreover, \eqref{e:R-n1-hat-liminf} implies that 
$$\liminf_{i\to\infty}  \wdh R_{n_{m_i},1}(y,z) \ge \liminf_{n\to\infty} \wdh R_{n,1}(y,z)  \ge f(y,z).$$ 
Thus it follows that for any $t \in \R^{+}$, $\e > 0$, and all sufficiently large $i\in \NN$, we have $\{f +R > t + \e\} \subset \{ \wdh R_{n_{m_i},1} + R > t\}$.  Hence the weak convergence of $\ovl\nu_{1,n_{m_i}}$ to $\ovl\nu$ and this inclusion for $i$ sufficiently large yield 
\begin{align*} 
\ovl\nu\{f + R> t+\e\} & \le \liminf_{i\to\infty} \ovl \nu_{1,n_{m_i}}\{f + R > t+\e\}  \le \liminf_{i\to\infty} \ovl \nu_{1,n_{m_i}}\{\wdh R_{n_{m_i},1} + R> t\};
\end{align*}  
thus the conditions of Lemma~2.1 of \cite{serf:82} are satisfied.  Using that lemma and \propref{c1hat-lsc}, it follows that    
\begin{eqnarray*}
\liminf_{n\to\infty}\int_{\Gamma_{5}} (\wdh R_{n,1}(y,z) + R)\, \ovl\nu_{1,n} (dy\times dz) &=& \lim_{i\to\infty}\int_{\Gamma_{5}} (\wdh R_{n_{m_i},1}(y,z) + R)\, \ovl\nu_{1,n_{m_i}} (dy\times dz) \\ 
&\ge& \int_{\Gamma_{5}} (f(y, z) + R)\, \ovl\nu(dy\times dz).
\end{eqnarray*} 
Recalling that $f = \wdh{BU}_0 + \wdh{c}_1 \geq 0$ and that $\ovl\nu_{1,n_{m_i}} \Rightarrow \ovl\nu$ implies convergence of the masses $\ovl\nu_{1,n_{m_i}}(\Gamma_5)$ to $\ovl\nu(\Gamma_5)$, this gives  
$$\liminf_{n\to\infty}\int_{\Gamma_{5}} \wdh R_{n,1}(y,z)\,  \ovl\nu_{1,n}(dy\times dz) \ge \int_{\Gamma_{5}} f(y, z)\, \ovl\nu(dy\times dz) \ge0.$$
Therefore \eqref{e:prop59E5-claim} is established, which combined with \eqref{double-sets}, completes the proof.
\end{proof}

Pulling all results together, we obtain our main theorem.

\begin{thm} \label{G0-feasible}
Assume Conditions \ref{diff-cnd} - \ref{cost-cnds}, \ref{extra-cnd} and \ref{AG0-unif-int} hold.  Let $(\tau,Z) \in {\cal A}_0$ with $J(\tau,Z) < \infty$.  Then
$$J(\tau,Z) \geq H_0^* = H_0(\yzstar,\zzstar) = J(\tau^*,Z^*)$$
in which $(\tau^*,Z^*)$ is the ordering policy \eqref{sS-tau-def} using an optimizing pair $(\yzstar,\zzstar) \in {\cal R}$.
\end{thm}

\begin{proof}
Let $(\tau,Z)\in {\cal A}_0$ satisfy $J(\tau,Z) < \infty$.  Let $X$ satisfy \eqref{controlled-dyn}, $\mu_{0,t}$ and $\nu_{1,t}$ be defined by \eqref{mus-t-def} for each $t > 0$.  Let $\{t_j\}$ be a sequence with $t_j \rightarrow \infty$ and
\begin{eqnarray} \label{cost-representation} \nonumber
J(\tau,Z) &=& \lim_{j\rightarrow \infty} \mbox{$\frac{1}{t_j}$} \EE\left[\int_0^{t_j} c_0(X(s))\, ds + \sum_{k=1}^\infty I_{\{\tau_k \leq t_j\}} c_1(X(\tau_k-),X(\tau_k))\right] \\
&=& \lim_{j\rightarrow \infty} \left(\int_{\overline{\cal E}} {c_0}(x)\, {\mu}_{0,t_j}(dx) + \int_{\overline{\cal R}} \wdh{c_1}(y,z)\, {\nu}_{1,t_j}(dy\times dz)\right).
\end{eqnarray}
The tightness of $\{{\mu}_{0,t_j}\}$ implies the existence of a weak limit ${\mu}_0$; without loss of generality, assume ${\mu}_{0,t_j} \Rightarrow {\mu}_0$ as $j\rightarrow \infty$.  \propref{J-finite} and its proof establish that 
$$\int_{\overline{\cal E}} {c_0}\, d {\mu}_0 \leq \liminf_{j\rightarrow \infty} \int_{\overline{\cal E}} {c_0}\, d {\mu}_{0,t_j} \leq J(\tau,Z) < \infty.$$

Since $U_n \in {\cal D}$, $\displaystyle\lim_{j\rightarrow \infty} \int_{\overline{\cal E}} AU_n\, d {\mu}_{0,t_j} = \int_{\overline{\cal E}} AU_n\, d {\mu}_0$.  \propref{bdd-adjoint} implies that for each $n$,
\begin{equation} \label{Gn-bar}
\lim_{j\rightarrow \infty} \left(\int_{\overline{\cal E}} AU_n(x)\,  {\mu}_0(dx) + \int_{\overline{\cal R}} \wdh{BU_n}(y,z)\,  {\nu}_{1,t_j}(dy\times dz)\right) = 0
\end{equation}
so adding \eqref{cost-representation} and \eqref{Gn-bar}  and taking the limit inferior as $n\rightarrow \infty$ yields, 
\begin{align*}
& J(\tau,Z) \\ &= \liminf_{n\rightarrow \infty} \lim_{j\rightarrow \infty} \left(\int_{\overline{\cal E}} (AU_n(x) + c_0(x))\,  {\mu}_{0,t_j}(dx) 
+ \int_{\overline{\cal R}} (\wdh{BU_n}(y,z) + \wdh{c_1}(y,z))\,  {\nu}_{1,t_j}(dy\times dz)\right) \\
&\geq \liminf_{n\rightarrow \infty} \liminf_{j\rightarrow \infty} \int_{\overline{\cal E}} (AU_n(x) + c_0(x))\,  {\mu}_{0,t_j}(dx) \\&  \qquad 
+\; \liminf_{n\rightarrow \infty} \liminf_{j\rightarrow \infty} \int_{\overline{\cal R}} (\wdh{BU_n}(y,z) + \wdh{c_1}(y,z))\,  {\nu}_{1,t_j}(dy \times dz) \\
&\geq \liminf_{n\rightarrow \infty} \int_{\overline{\cal E}} (AU_n(x) + c_0(x))\,  {\mu}_0(dx)  
+\; \liminf_{n\rightarrow \infty} \liminf_{j\rightarrow \infty} \int_{\overline{\cal R}} (\wdh{BU_n}(y,z) + \wdh{c_1}(y,z))\,  {\nu}_{1,t_j}(dy \times dz) \\
&\geq H_0^*;
\end{align*}
Propositions \ref{AGn-c0-rel} and \ref{BGn-c1-rel} establish the last inequality.
\end{proof}

\section{Examples} \label{sect:examples}
We begin by briefly discussing the inventory management models in \cite{helm:18}.  This paper shows that optimality of a $(\yzstar,\zzstar)$ policy extends to models having deficient supply.  The main example (in \sectref{LO}) demonstrates the efficacy of this optimization approach for a more complicated stochastic logistic inventory model having nearly proportional yields.

\subsection{Drifted Brownian motion inventory models} \label{BMIM}
The first inventory problem considers the classical fundamental process of a drifted Brownian motion $X_0$ satisfying the stochastic differential equation
\begin{equation} \label{cost-61}
dX_0(t) = - \mu\, dt + \sigma\, dW(t), \qquad X_0(0)=x_0,$$
in which $\mu, \sigma > 0$ and $W$ is a standard Brownian motion, under the cost structure 
$$c_0(x) = \left\{\begin{array}{cl}
-c_b\, x, & \quad x < 0, \\
c_h\, x, & \quad x\geq 0 
\end{array} \right. \quad \mbox{and} \quad c_1(y,z) = k_1 + k_2(z-y), \quad -\infty < y \leq z < \infty,
\end{equation}
with $c_b, c_h, k_1, k_2 > 0$. 

A modification of the problem has reflection at $0$ so that no backordering is allowed with the cost structure
$$c_0(x) = k_3 x + k_4 e^{-x} \quad \mbox{for } x\geq 0 \qquad \mbox{and} \qquad c_1(y,z) = k_1 + k_2 \sqrt{z-y} \quad \mbox{for } 0 \leq y \leq z < \infty,$$
again with $k_1, k_2, k_3, k_4 > 0$.

As mentioned previously, \cndref{diff-cnd} is the same in both papers and Condition~2.2 of \cite{helm:18} is the same as \cndref{cost-cnds} in this paper.  Further, Condition~2.3 of the previous paper is more restrictive than \cndref{extra-cnd} here.  Thus Conditions~\ref{diff-cnd}, \ref{cost-cnds} and \ref{extra-cnd} are satisfied by both of these models, as established in the 2018 paper.  Thus for any family ${\cal Q}$ satisfying \cndref{Q-cnds}, the conditions of \thmref{F-optimizers} are satisfied and there exists an optimizing pair $(\yzstar,\zzstar) \in {\cal R}$ of $H_0$. 

Turning to \thmref{G0-feasible} to establish the optimality of the $(\yzstar,\zzstar)$ policy, \cndref{AG0-unif-int} of this paper differs from Condition~5.1 of \cite{helm:18} only in the use of $U_0 = g_0 - H_0^* \zeta$ in place of $G_0 = g_0 - F_0^* \zeta$.  The verification of Condition~5.1 of the previous paper does not rely on $F_0^*$.  Thus the same argument using $U_0$ in place of $G_0$ demonstrates that \cndref{AG0-unif-int} holds for both problems involving the drifted Brownian motion  model.  \thmref{G0-feasible} therefore establishes that the $(\yzstar,\zzstar)$ ordering policy is optimal.

\subsection{Geometric Brownian motion storage models}
The second model examined in \cite{helm:18} takes its fundamental dynamics to be a geometric Brownian motion process satisfying the stochastic differential equation

$$dX_0(t) = -\mu X_0(t)\, dt + \sigma X_0(t)\, dW(t), \qquad X_0(t) = x_0 \in (0,\infty),$$
in which $\mu, \sigma > 0$.  Two different cost structures were analyzed:
$$c_0(x) = k_3 x + k_4 x^\beta \quad \mbox{for } 0 < x < \infty \qquad \mbox{and} \qquad c_1(y,z) = k_1 + k_2\sqrt{z-y} \quad \mbox{for } 0 < y \leq z < \infty;$$
and
\begin{eqnarray*}
c_0(x) &=& \left\{\begin{array}{cl}
k_3 (1-x), & \qquad \qquad \qquad \qquad \mbox{for } 0 < x < 1, \\
k_4(x-1), & \qquad \qquad \qquad \qquad \mbox{for } 1 \leq x < \infty, 
\end{array}\right. \\
c_1(y,z) &=& k_1 + \mbox{$\frac{1}{2}$}(y^{-\frac{1}{2}} - z^{-\frac{1}{2}}) + \mbox{$\frac{1}{2}$} (z-y) \quad \; \; \mbox{for } 0 < y \leq z < \infty,
\end{eqnarray*}
in which the parameters $k_1, k_2, k_3, k_4 > 0$ and $\beta < 0$.

For the geometric Brownian motion model, Conditions~\ref{diff-cnd}, \ref{cost-cnds} and \ref{extra-cnd} are shown to be satisfied in the 2018 paper.   Thus for any family ${\cal Q}$ satisfying \cndref{Q-cnds}, \thmref{F-optimizers} establishes the existence of an optimizing pair $(\yzstar,\zzstar) \in {\cal R}$ of $H_0$.  Furthermore, similarly to the drifted Brownian motion model, \cndref{AG0-unif-int} follows from the same analysis as in the proof of Theorem~6.4 of our 2018 paper with $U_0$ replacing $G_0$.  Therefore \thmref{G0-feasible} shows that $(\yzstar,\zzstar)$ ordering policy is optimal for deficient supply models.

\subsection{Logistic storage model} \label{LO}
Our third example is a logistic inventory model  in a random environment with a special family of random supplies. The process is an adaptation to an inventory set-up of a population model analyzed by \cite{lungu:97} in the context of a particular harvesting study.

For this model, the inventory level of a product (in the absence of orders) satisfies the stochastic differential equation
\begin{equation} \label{dyn_logistic}
dX_0(t) =  - \mu X_0(t) (k-X_0(t))\, dt + \sigma X_0(t) (k-X_0(t))\, dW(t), \qquad X_0(0) = x_0,
\end{equation}
in which $k$, $\mu$ and $\sigma$ are positive constants.  Set $\beta := - \frac{2 \mu }{k \sigma^2}$ and require $\beta < -1$. The process $X_0$ evolves on the bounded state space ${\mathcal I} = (0,k)$.  With reference to Chapter 15 of \cite{karl:81}, straightforward calculations verify that this model satisfies \cndref{diff-cnd}.  In particular, both endpoints are natural, $0$ is attracting and $k$ is non-attracting; see also \cite{hell:95}.  In comparison with  geometric Brownian motion, both boundary points are finite.  We identify the scale function and speed measure in \eqref{scale} and \eqref{speed} for a particular scaling of the logistic model.


A common yield structure when there are deficient supplies is provided by the uniform distribution on $(y,z)$, representing proportional yields.  When $\I$ is unbounded above, this family of uniform distributions on $(y,z)$ for $y,z \in \I$ is easily seen to satisfy \cndref{Q-cnds}(c) since the mass escapes to $\infty$ as $z\to\infty$.   Unfortunately this condition is no longer true for a uniform distribution with $y$ fixed and $z \to k$ for this example since the right boundary is a finite value. Thus, we adopt the famly of {\em `$z$-skewed uniform distributions'}\/ as a surrogate, resulting in a model with nearly proportional yields. 

To be precise, choose a large integer $j$ and for each $(y, z) \in  {\mathcal R}$, let $Q(\cdot \,;y,z)$ be the uniform distribution on the interval having {\em left endpoint} $(1-(z/k)^j) y + (z/k)^j z$ and  {\em right endpoint} $z$. In this choice, the left enpoint is a convex combination of $y$ and $z$ with a weight factor $(z/k)^j$ that more heavily favors $z$ as $z$ approaches the upper boundary $k$.  Clearly, this family of distributions satisfies the ASC condition as well as the MDG condition in \cndref{Q-cnds}(a,ii).  Furthermore, depending on the choice of $j$, the measure $Q(\cdot\,; y,z)$ is a `reasonable' approximation to the uniform 
distribution on $(y,z)$ when $z$ is not too close to $k$.  Therefore this family of random effects distributions results in a model having nearly proportional yields.  Finally, we take $Q(\cdot\,,y,y) = \delta_y(\cdot)$ so that \cndref{Q-cnds}(a,i) holds and it is easy to verify the weak convergence of the measures in \cndref{Q-cnds}(b).
 
 For this example, we choose the bounded holding cost function $c_0(x) := k_0 (x-{\bar x})^2$ for $0 < x < k$, in which $k_0$ is a positive constant and the number ${\bar x} \in (0,k)$ characterizes a `preferred' inventory level.  Further, we choose the order cost function $c_1(y,z)$ in \eqref{cost-61}.  Again, straightforward analysis verifies \eqref{c0-M-integrable} and hence  \cndref{cost-cnds} is satisfied.  

Scaling the inventory process by the factor $k$ and adjusting the parameters appropriately, we can set $k = 1$ without loss of generality.  The scale function $S$ and the speed measure $M$ associated with $X_0$ can be determined as follows.  Let $C_1 = ({x_0/(1-x_0)})^{\beta}$,  $C_2 = 1/(\sigma^2C_1) = ((1-x_0)/x_0)^{\beta}/\sigma ^2$, and let ${_2F_1}$ denote the (Gaussian) hypergeometric function. Let $\tilde S(x) = C_1  x^{(1-\beta)}/(1-\beta)\, {_2F_1}(1-\beta,-\beta;2-\beta;x)$.  Then, 
\begin{equation} \label{scale}
S(x) = C_1  \int_{x_0}^{x} {((1-u)/u)}^{\beta} \, du =    \, \,{\tilde S(x)} - {\tilde S(x_0)},\qquad 0 < x < 1,
\end{equation}
while $M[a,b] = \int_{a}^{b}m(v) dv$ for any $[a, b] \subset (0,1)$, where the speed density $m$ is given by  
\begin{equation} \label{speed}
m(v) = C_2 (1-v)^{-(\beta+2)}{v}^{\beta - 2}, \qquad 0 < v < 1.
\end{equation}
For later reference, we note that $S'(x) = C_1 (\frac{1-x}{x})^\beta$ for $x\in (0,1)$.

Now turning to \cndref{extra-cnd}, since each boundary is natural, we need to check that there is some $(y,z)\in \R$ for which $H_0(y,z)=(\wdh{Bc_1}(y,z)+\wdh{Bg_0}(y,z))/\wdh{B\zeta}(y,z)$ is smaller than the holding cost rates at the boundaries.  Toward this end, the expressions for $\zeta$ and $g_0$ simplify considerably when we set $x_0 = {\bar x} = 1/2$ so we make this selection for this illustration.  These functions are 
\begin{align} \label{zeta-63}
\zeta(x) & = -\,\frac{2 \left(1-2 x+2 \beta  \ln(2-2 x)+\beta  (1+\beta ) \ln\left(\frac{x}{1-x}\right)\right)}{\sigma ^2 \beta  \left(-1+\beta^2\right)},\\ \label{g0-63}
g_0(x) &= \frac{k_0 \left((-1+2 x) \left(-1+2 \beta ^2\right)-2 \beta  \ln(2-2 x)-\beta  (1+\beta ) \ln\left(\frac{x}{1-x}\right)\right)}{2\sigma ^2 \beta  \left(-1+\beta ^2\right)}\ .
\end{align}
The functions $\wdh{Bc_1}$, $\wdh{B\zeta}$ and $\wdh{Bg_0}$ are then obtained by integrating the functions given above with respect to the measures $Q$. Usually, this integration is best accomplished using 
 software packages such as Maple or Mathematica since the formulas become messy.  Then by elementary but rather lengthy calculations one verifies \cndref{extra-cnd}.  

For more general parameters in this model, \eqref{zeta-63} and \eqref{g0-63} become more involved and even become analytically intractable for different families of random effects measures.  An alternative approach to verifying \cndref{extra-cnd} is to simply optimize $H_0$ and then compare the optimal value $H_0^*$ with the cost rates $c_0(0) = k_0/4 = c_0(1)$.  An optimizing pair $(y^*,z^*)$ in the interior would then satisfy \cndref{extra-cnd} for this model when $H_0^* < k_0/4$.   Minimizing $H_0$ is a two-dimensional optimization problem.  Since \cndref{extra-cnd} only requires the existence of a pair $(y,z)\in {\cal R}$, other alternatives for verifying this condition would be: (i) to fix one of the variables or a relation between the variables,  perform a one-dimensional optimization and compare this value of $H_0$ against $k_0/4$; or (ii) to compare the values of $H_0$ from a random search of $\cal R$.  Each of these alternate approaches is numerical, rather than analytic.

Finally, to see that an $(s,S)$-policy is optimal we need to verify Condition \ref{AG0-unif-int} (a, ii) and (b, ii). To this end, recall $U_0(x) = g_0(x) - H_{0}^{*} \zeta(x)$ and observe that $c_{0}(x) - H_{0}^{*}$ is uniformly bounded on the unit interval. With $\zeta$ and $g_0$ given in \eqref{g0-fn} and using the  expressions of the scale density \eqref{scale} and the speed density \eqref{speed} we have
\begin{align} \label{e:kurtsig_U0'estimate:ex3} \nonumber 
 |\sigma x (1 - x) U'_0(x)|  &  \le     |\sigma x (1 - x) S'(x)| \int_{x}^{1} |c_0(v) - H_{0}^{*}| dM(v) \\ 
  & \le    K x^{1-\beta}(1-x)^{1+\beta}\int_{x}^{1} v^{\beta-2} (1-v)^{-\beta-2} dv,
\end{align} 
where $K$ is a positive constant independent of $x$ or $x_{0}$. 
To see that the left hand side of  \eqref{e:kurtsig_U0'estimate:ex3} is uniformly bounded on $[0,1]$  which, in turn, implies   Condition \ref{AG0-unif-int} (a, ii) and (b, ii), it is clearly sufficient to find bounds in some neighborhoods of the two endpoints.  The simple idea is to verify that: (i) for $x$ close to $1$, the integral on the right hand side of the inequalities decreases at the same rate as the factor $(1-x)^{1+\beta}$ increases; and (ii) when $x$ is close to zero, the integral increases at a rate no faster than the rate at which the factor  $x^{1-\beta}$ decreases.  

(i) For $x\in (\frac12, 1)$ the integral in \eqref{e:kurtsig_U0'estimate:ex3} is dominated by
$$\int_{x}^{1} v^{\beta-2} (1-v)^{-\beta-2} dv  \quad  \le  \quad 2^{2-\beta} \int_{x}^{1}(1-v)^{-\beta-2} dv  \quad \le  \quad \mbox{$\frac{2^{2-\beta}}{-\beta-1}$} (1-x)^{-\beta-1}$$
and hence the left-hand side of \eqref{e:kurtsig_U0'estimate:ex3} is bounded by some $K_1$ for $x\in (\frac12,1)$.   

(ii) Similarly for $x\in (0, \frac12)$, a dominating function for the integral in \eqref{e:kurtsig_U0'estimate:ex3} is determined as follows:
\begin{eqnarray*}
\int_{x}^{1} v^{\beta-2} (1-v)^{-\beta-2} dv &=& \int_{x}^{1/2} v^{\beta-2} (1-v)^{-\beta-2} dv + \int_{1/2}^{1} v^{\beta-2} (1-v)^{-\beta-2} dv \\
&\leq&(2^{2+\beta}\vee 1)  \int_{x}^{1/2} v^{\beta-2}  dv + K_2 \\
&=& \mbox{$\frac{2^{2+\beta}\vee 1}{1-\beta}$}\, x^{\beta-1} + K_3,
\end{eqnarray*}
in which $K_2$ is the value of the integral over $[\frac12,1]$ and $K_3$ then adjusts this value by the contribution of the first integral at the boundary $1/2$.   
Thus, taking into account the factor $x^{1-\beta}$ on the right-hand side of \eqref{e:kurtsig_U0'estimate:ex3}, the left-hand side of \eqref{e:kurtsig_U0'estimate:ex3} is bounded for $x\in (0,1/2)$.

Using both estimates in \eqref{e:kurtsig_U0'estimate:ex3} together with the fact that $\lim_{x\to 0} \sigma x (1-x) U'_0(x)$ exists and is finite  we have thus shown that $|\sigma x (1-x) U'_0(x)|$ is uniformly bounded on $[0,1]$.  Since the denominators are bounded below by $1$, \cndref{AG0-unif-int} holds. 

In summary, the model satisfies Conditions~\ref{diff-cnd}, \cndref{Q-cnds}, \ref{cost-cnds} and \ref{extra-cnd}.  Therefore \thmref{F-optimizers} establishes the existence of an optimizing pair $(\yzstar,\zzstar) \in {\cal R}$ of $H_0$.  Furthermore, since \cndref{AG0-unif-int} holds, \thmref{G0-feasible} shows that the $(\yzstar,\zzstar)$ ordering policy is optimal for this particular logistic inventory model.
  
Finally, we numerically illustrate the effect of using the optimization results in this paper for a particular set of parameters.  For comparison purposes, three models based on the logistic dynamics in \eqref{dyn_logistic} are examined.  Model 1 assumes no noise by setting $\sigma=0$ so that the dynamics are deterministic, and uses the non-deficient supply measures $Q(\,\cdot\,;y,z) = \delta_{\{z\}}(\,\cdot\,)$ for all $(y,z) \in {\cal R}$.  Model 2 has $\sigma=1/10$ resulting in random fluctuations in the inventory level but also uses $Q(\,\cdot\,;y,z) = \delta_{\{z\}}(\,\cdot\,)$ for all $(y,z) \in {\cal R}$ so that the amount ordered is the amount delivered.  Model 3 takes $\sigma=1/10$ and uses the nearly proportional yield transition functions $Q$ defined earlier in this subsection, with $j=10$.  The other parameters in this illustration are $k=1$, $\mu=1/20$, $k_0=100$, $k_1=9$, $k_2=4$ and $x_0 = {\bar x} = 1/2$.

Table \ref{table} illustrates the impact a random environment and/or random supplies have on optimal characteristics of the logistic inventory model.  Specifically, the following characteristics of the optimal solutions have been computed:
\begin{itemize}
\item the order `From' level $\yzstar$ and the deterministic order `To' or nominal order `To' level $\zzstar$;
\item the `Mean Supply', a deterministic quantity in Models 1 and 2;
\item the optimal expected long-run average `Cost'; and
\item the `Mean Cycle Length'; the cycle length is again deterministic for Model 1.
\end{itemize}
Observe that the optimal value of $H_0^* = 1.33092 = H_0(0.384973,0.6575) < 25 = k_0/4$ so \cndref{extra-cnd} is satisfied.
\begin{table}[h!]
    \label{tab:table1}
     \begin{center}
    \begin{tabular}{|l||c|c|c|c|c|} 
    \hline
      \textbf{Model} & \textbf{From} & \textbf{To} & \textbf{Mean Supply} & \textbf{Cost} & \textbf{Mean Cycle Length}\\
      \hline
      Model 1 & 0.40567 & 0.59433 & 0.188661 & 0.938043 &  15.2759\\
      Model 2 & 0.381724 & 0.56993 & 0.188206 & 1.00067 &  15.2779\\
      Model 3 & 0.384973 & 0.6575 & 0.138321 & 1.33092 &  11.2843\\ \hline
    \end{tabular}
  \end{center}
    \caption{Comparison of Three Logisitic Inventory Models. \label{table}}
\end{table}

From a management point of view the following observations are important.  The nearly proportional yield model having random fluctuations in inventory result in cost increases of 42\% and 33\% over Models 1 and 2, respectively.  Also, the uncertainty of the environment and the fluctuating deliveries typically shorten the mean cycle length, despite the nominal order interval increasing in length as randomness is added to the process and to the delivered amounts.  Thus, ordering tends to occur more frequently for the stochastic models.   

 Additional insights into the characteristics of the optimal nominal policy and optimal inventory process can be obtained by more extensive sensitivity analysis.  For instance, for modifications of this example, various statistics of the aforementioned quantities such as the mean cycle time, as well as other quantities,  can be computed or derived from simulation studies.  

As indicated earlier, uniqueness of the optimal policy is not analytically guaranteed.  However, one may obtain contour plots of $H_0$ numerically and thereby   determine the uniqueness of the optimal policy for this particular model and for more general stochastic differential equations and $Q$ distributions.

\end{document}